\documentclass[a4paper,final]{amsart}

\usepackage{amsmath, amssymb, amsthm}
\usepackage{bm}
\usepackage{amsfonts}
\usepackage{enumitem,fullpage}
\usepackage{showkeys}
\usepackage{contour}
\contournumber{64}
\contourlength{.06em}
\usepackage{color}
\usepackage{here}
\usepackage{tikz}
\usetikzlibrary{intersections, calc, arrows, cd,positioning,patterns}
\usepackage[colorlinks,citecolor=darkgreen,linkcolor=red]{hyperref}
\definecolor{darkgreen}{rgb}{0.0, 0.6, 0.0}

\setenumerate[1]{label={\upshape(\arabic*)}}
\setenumerate[2]{label={\upshape(\alph*)}}
\numberwithin{equation}{section}
\numberwithin{figure}{section}

\allowdisplaybreaks

\theoremstyle{plain}
\newtheorem{thm}{Theorem}[section]
\newtheorem{prp}[thm]{Proposition}
\newtheorem{lem}[thm]{Lemma}
\newtheorem{cor}[thm]{Corollary}
\newtheorem{fct}[thm]{Fact}
\newtheorem{ques}[thm]{Question}
\newtheorem*{thm*}{Theorem}
\newtheorem*{prp*}{Propostion}
\newtheorem*{lem*}{Lemma}
\newtheorem*{fct*}{Fact}

\theoremstyle{definition}
\newtheorem{dfn}[thm]{Definition}

\newtheorem{rmk}[thm]{Remark}
\newtheorem{ex}[thm]{Example}

\newtheorem*{dfn*}{Definition}
\newtheorem*{nt*}{Notattion}
\newtheorem*{rmk*}{Remark}
\newtheorem*{ex*}{Example}

\newtheorem{thma}{Theorem}

\newcommand{\K}{\mathsf{K}}
\newcommand{\M}{\mathsf{M}}
\newcommand{\D}{\mathsf{D}}
\newcommand{\gp}{\mathsf{gp}}
\newcommand{\pow}{\mathsf{P}}

\newcommand{\SemiGrp}{\mathsf{SemiGrp}}

\newcommand{\gen}[1]{\left\la #1 \right\ra}
\newcommand{\ideal}{\mathrm{ideal}}
\newcommand{\face}{\mathrm{face}}

\newcommand{\spcl}{\mathrm{spcl}}
\newcommand{\Bor}{\bigvee}
\newcommand{\Band}{\bigwedge}

\DeclareMathOperator{\Serre}{Serre}
\DeclareMathOperator{\Face}{Face}

\DeclareMathOperator{\MSpec}{MSpec}
\DeclareMathOperator{\LSpec}{LSpec}

\DeclareMathOperator{\Spcl}{Spcl}
\DeclareMathOperator{\Genl}{Genl}
\newcommand{\genl}{\mathrm{genl}}

\newcommand{\irred}{\mathrm{irred}}

\newcommand{\Cross}{\mathbin{\tikz [x=1.4ex,y=1.4ex,line width=.2ex] \draw (0,0) -- (1,1) (0,1) -- (1,0);}}%

\newcommand{\ol}{\overline}

\newcommand{\la}{\langle}
\newcommand{\ra}{\rangle}
\newcommand{\iso}{\cong}

\newcommand{\imply}{\Rightarrow}
\newcommand{\equi}{\Leftrightarrow}

\newcommand{\xr}[1]{\xrightarrow{\, #1 \, }}
\newcommand{\inj}{\hookrightarrow}
\newcommand{\surj}{\twoheadrightarrow}

\newcommand{\simto}{\xr{\sim}}
\newcommand{\isoto}{\xr{\iso}}

\newcommand{\infl}{\rightarrowtail}
\newcommand{\defl}{\twoheadrightarrow}

\newcommand{\bbF}{\mathbb{F}}

\newcommand{\bbN}{\mathbb{N}}
\newcommand{\bbP}{\mathbb{P}}

\newcommand{\bbZ}{\mathbb{Z}}

\newcommand{\bbk}{\Bbbk}
\newcommand{\calB}{\mathcal{B}}

\newcommand{\LL}{\Lambda}

\newcommand{\lmd}{\lambda}

\newcommand{\sg}{\sigma}

\newcommand{\symg}{\mathfrak{S}}

\DeclareMathOperator{\rk}{rk}

\DeclareMathOperator{\Supp}{Supp}
\newcommand{\abs}[1]{\left| #1 \right|}

\DeclareMathOperator{\Set}{\mathsf{Set}}

\DeclareMathOperator{\Mon}{\mathsf{Mon}}

\newcommand{\catA}{\mathcal{A}}

\newcommand{\catC}{\mathcal{C}}
\newcommand{\catD}{\mathcal{D}}
\newcommand{\catE}{\mathcal{E}}
\newcommand{\catF}{\mathcal{F}}

\newcommand{\catS}{\mathcal{S}}

\newcommand{\catX}{\mathcal{X}}
\newcommand{\catY}{\mathcal{Y}}

\DeclareMathOperator{\Hom}{Hom}
\DeclareMathOperator{\End}{End}

\DeclareMathOperator{\Ima}{Im}
\DeclareMathOperator{\Ker}{Ker}
\DeclareMathOperator{\Cok}{Cok}

\DeclareMathOperator{\Mod}{\mathsf{Mod}}
\DeclareMathOperator{\catmod}{\mathsf{mod}}

\DeclareMathOperator{\proj}{\mathsf{proj}}

\DeclareMathOperator{\simp}{\mathsf{sim}}

\DeclareMathOperator{\noeth}{\mathsf{noeth}}

\DeclareMathOperator{\Spec}{Spec}

\DeclareMathOperator{\Ann}{Ann}
\newcommand{\mm}{\mathfrak{m}}

\newcommand{\pp}{\mathfrak{p}}

\DeclareMathOperator{\Qcoh}{\mathsf{Qcoh}}

\DeclareMathOperator{\coh}{\mathsf{coh}}
\DeclareMathOperator{\tor}{\mathsf{tor}}
\DeclareMathOperator{\vect}{\mathsf{vect}}

\newcommand{\shF}{\mathcal{F}}
\newcommand{\shG}{\mathcal{G}}
\newcommand{\shH}{\mathcal{H}}
\newcommand{\shI}{\mathcal{I}}

\newcommand{\shL}{\mathcal{L}}
\newcommand{\shM}{\mathcal{M}}
\newcommand{\shN}{\mathcal{N}}
\newcommand{\shO}{\mathcal{O}}

\newcommand{\shT}{\mathcal{T}}
\newcommand{\shU}{\mathcal{U}}
\newcommand{\shV}{\mathcal{V}}

\newcommand{\shHom}{\mathcal{H}om}

\DeclareMathOperator{\Div}{Div}

\DeclareMathOperator{\Pic}{Pic}

\newcommand{\arr}[1]{\arrow[{#1}]}
\newenvironment{bsmatrix}{\left[\begin{smallmatrix}}{\end{smallmatrix}\right]}
\newenvironment{enur}{\begin{enumerate}[label={\upshape(\roman*)}]}{\end{enumerate}}
\newenvironment{enua}{\begin{enumerate}[label={\upshape(\arabic*)}]}{\end{enumerate}}

\tikzset{
  symbol/.style={
    draw=none,
    every to/.append style={
      edge node={node [sloped, allow upside down, auto=false]{$#1$}}}
  }
}

\makeatletter
\newcommand*{\da@rightarrow}{\mathchar"0\hexnumber@\symAMSa 4B }
\newcommand*{\da@leftarrow}{\mathchar"0\hexnumber@\symAMSa 4C }
\newcommand*{\xdashrightarrow}[2][]{%
  \mathrel{%
    \mathpalette{\da@xarrow{#1}{#2}{}\da@rightarrow{\,}{}}{}%
  }%
}
\newcommand{\xdashleftarrow}[2][]{%
  \mathrel{%
    \mathpalette{\da@xarrow{#1}{#2}\da@leftarrow{}{}{\,}}{}%
  }%
}
\newcommand*{\da@xarrow}[7]{%
  \sbox0{$\ifx#7\scriptstyle\scriptscriptstyle\else\scriptstyle\fi#5#1#6\m@th$}%
  \sbox2{$\ifx#7\scriptstyle\scriptscriptstyle\else\scriptstyle\fi#5#2#6\m@th$}%
  \sbox4{$#7\dabar@\m@th$}%
  \dimen@=\wd0 %
  \ifdim\wd2 >\dimen@
    \dimen@=\wd2 %
  \fi
  \count@=2 %
  \def\da@bars{\dabar@\dabar@}%
  \@whiledim\count@\wd4<\dimen@\do{%
    \advance\count@\@ne
    \expandafter\def\expandafter\da@bars\expandafter{%
      \da@bars
      \dabar@ 
    }%
  }%
  \mathrel{#3}%
  \mathrel{%
    \mathop{\da@bars}\limits
    \ifx\\#1\\%
    \else
      _{\copy0}%
    \fi
    \ifx\\#2\\%
    \else
      ^{\copy2}%
    \fi
  }%
  \mathrel{#4}%
}

\newcommand{\colim@}[2]{%
  \vtop{\m@th\ialign{##\cr
    \hfil$#1\operator@font colim$\hfil\cr
    \noalign{\nointerlineskip\kern1.5\ex@}#2\cr
    \noalign{\nointerlineskip\kern-\ex@}\cr}}%
}
\newcommand{\colim}{%
  \mathop{\mathpalette\colim@{\rightarrowfill@\scriptscriptstyle}}\nmlimits@
}
\newcommand{\plim}{%
  \mathop{\mathpalette\varlim@{\leftarrowfill@\scriptscriptstyle}}\nmlimits@
}
\makeatother
\newcommand{\red}[1]{{\color{magenta}#1}} 
\begin{document}

\title{The spectrum of Grothendieck monoid: \\%
classifying Serre subcategories and reconstruction theorem}

\author{Shunya Saito}
\address{Graduate School of Mathematics, Nagoya University, Chikusa-ku, Nagoya. 464-8602, Japan}
\email{m19018i@math.nagoya-u.ac.jp}


\begin{abstract}
The Grothendieck monoid of an exact category is a monoid version of the Grothendieck group.
We use it to classify Serre subcategories of an exact category
and to reconstruct the topology of a noetherian scheme.
We first construct bijections between 
(i) the set of Serre subcategories of an exact category,
(ii) the set of faces of its Grothendieck monoid,
and (iii) the monoid spectrum of its Grothendieck monoid.
By using (ii), we classify Serre subcategories of exact categories 
related to a finite dimensional algebra and a smooth projective curve.
For this,
we determine the Grothendieck monoid of the category of coherent sheaves on 
a smooth projective curve.
By using (iii), we introduce a topology on the set of Serre subcategories.
As a consequence, we recover the topology of a noetherian scheme
from the Grothendieck monoid.
\end{abstract}

\maketitle
\tableofcontents
\section{Introduction}\label{s:Intro}
\subsection{Background}
Classifying nice subcategories of an abelian category or a triangulated category is
quite an active subject which has been studied in various areas of mathematics
such as representation theory and algebraic geometry.
A typical example is the following result given by Gabriel:
\begin{fct}[{\cite[Proposition VI.2.4]{Gab62}}]\label{fct:Gab 1}
Let $X$ be a noetherian scheme.
There are inclusion-preserving bijections between the following sets:
\begin{itemize}
\item
The set of Serre subcategories $\catS$ of the category $\coh X$ of coherent sheaves on $X$.
\item
The set of specialization-closed subsets $Z$ of $X$.
\end{itemize}
Here the assignments are given by
\[
\catS \mapsto \Supp\catS :=\bigcup_{\shF \in\catS}\Supp \shF, \quad
Z \mapsto \coh_Z X:=\{\shF \in \coh X \mid \Supp \shF \subseteq Z\}.
\]
\end{fct}
In \cite{Gab62}, Gabriel also proved that any noetherian scheme $X$ can be reconstructed
from the category $\Qcoh X$ of quasi-coherent sheaves on $X$.
These results have been generalized in several ways.
See \cite{Zi84,He97,Ka12} for classifications of Serre subcategories
and \cite{Ro98,BO01,Balmer,BKS07,GP08,Br16} for reconstruction theorems.
The present paper sheds new light on these results by using \emph{Grothendieck monoids}.

The Grothendieck monoid $\M(\catE)$ is a monoid version of the Grothendieck group,
which is defined for each exact category $\catE$.
Several authors studied the Grothendieck monoid and extract information
that the Grothendieck group does not contain.
See \cite{Br97,Br98,Br03} for module categories,
\cite{En22} for exact categories related to finite dimensional algebras,
and \cite{BG16} for an application to Ringel-Hall algebras.
This paper aims to study Serre subcategories of an exact category
via its Grothendieck monoid.
\subsection{Main results}
Let us introduce the main subject in this paper.
An additive subcategory $\catS$ of an exact category $\catE$ is \emph{Serre} 
if for any conflation $0\to X \to Y\to Z\to 0$ in $\catE$,
we have that $Y\in\catS$ if and only if both $X\in\catS$ and $Z\in\catS$.
The reader unfamiliar with the language of exact categories
can think of exact categories and conflations 
as abelian categories and short exact sequences, respectively.
The following result is the starting point of this paper.
\begin{thma}[{Proposition \ref{prp:Serre face prime}}]\label{thm:A}
For an exact category $\catE$,
there are bijections between the following sets:
\begin{enua}
\item
The set $\Serre(\catE)$ of Serre subcategories of $\catE$.
\item
The set $\Face(\M(\catE))$ of faces of the Grothendieck monoid $\M(\catE)$.
\item
The set $\MSpec \M(\catE)$ of prime ideals of the Grothendieck monoid $\M(\catE)$.
\end{enua}
\end{thma}
Let us explain the terminologies used above.
Let $M$ be a commutative monoid.
A non-empty subset $F$ of $M$ is a \emph{face}
if for all $x,y\in M$, we have that $x+y\in F$ if and only if both $x\in F$ and $y\in F$.
A subset $\pp$ of $M$ is a \emph{prime ideal}
if $\pp^c:=M\setminus \pp$ is a face of $M$.

The second set $\Face(\M(\catE))$ can be computed purely algebraically,
and its computation is much easier than 
examining the whole structure of the exact category $\catE$.
The third set $\MSpec \M(\catE)$ has a topology,
which is a natural analogy of the Zariski topology
on the spectrum $\Spec R$ of a commutative ring $R$.
These lead us in two directions.

The one direction is a classification of Serre subcategories
by using faces of the Grothendieck monoid.
We propose the following strategy to classify Serre subcategories of an exact category $\catE$.
\begin{enua}
\item
Relate the Grothendieck monoid $\M(\catE)$ with an abstract monoid $M$.
\item
Classify faces of the abstract monoid $M$.
\item
Classify Serre subcategories of $\catE$ by using (1) and (2).
\end{enua}
Following this strategy, we classify Serre subcategories of some explicit exact categories.
We present a primitive example to illustrate this idea.

\begin{ex}
Consider the category $\vect \bbk$ of finite dimensional vector spaces over a field $\bbk$.
We will classify extension-closed subcategories of $\vect \bbk$ by using a monoid.
Here a subcategory $\catX$ of $\vect \bbk$ is \emph{extension-closed} if
it contains a zero object and for any exact sequence $0\to U\to V\to W\to 0$,
the condition $U,W \in\catX$ implies $V\in\catX$.
Let $\abs{\vect \bbk}$ be the set of isomorphism classes of objects of $\vect \bbk$.
The assignment $V \mapsto \dim_{\bbk} V$,
where $\dim_{\bbk} V$ denotes the dimension of $V$ over $\bbk$,
gives a bijection $\abs{\vect \bbk} \isoto \bbN$.
Here $\bbN$ is the set of non-negative integers.
We identify the (full) subcategories of $\vect \bbk$ (closed under isomorphisms)
with the subsets of $\bbN$ by this bijection.
For vector spaces $U$, $V$ and $W$,
there is an exact sequence $0\to U\to V\to W\to 0$ 
if and only if $\dim_{\bbk}V= \dim_{\bbk}U + \dim_{\bbk}W$.
Thus a subcategory of $\vect \bbk$ is extension-closed
if and only if the corresponding subset of $\bbN$ is a submonoid.
Therefore, classifying extension-closed subcategories of $\vect \bbk$ 
and submonoids of $\bbN$ are equivalent. 
This example shows why and how monoids are used to classify subcategories.
We generalize this story to any exact category in \S \ref{ss:Serre face}.
At the moment, we only mention that $\M(\vect \bbk) \iso \bbN$.
\end{ex}

In general, computation of the Grothendieck monoid $\M(\catE)$ is difficult
even if $\catE$ is an extension-closed subcategory of the category $\catmod\LL$
of finitely generated modules over a finite dimensional algebra $\LL$ (cf. \cite[Section 7]{En22}).
However, we determine the Grothendieck monoids
of the following exact categories related to a smooth projective curve $C$:
\begin{itemize}
\item
The category $\coh C$ of coherent sheaves on $C$,
\item
The category $\vect C$ of vector bundles over $C$,
\item 
The category $\tor C$ of coherent torsion sheaves on $C$.
\end{itemize}

\begin{thma}[{Proposition \ref{prp:M tor}, \ref{prp:M vect} and \ref{prp:M coh}}]\label{thm:B}
Let $C$ be a smooth projective curve over an algebraically closed field $\bbk$.
\begin{enua}
\item
$\M(\tor C)\iso\Div^+ C$ holds,
where $\Div^+ C$ is the monoid of effective divisors on $C$.
\item
$\M(\vect C)\iso \left(\Pic C \times \bbN^+\right) \cup\{(\shO_C,0)\}
\subseteq \Pic C \times \bbZ$ holds,
where $\Pic(C)$ is the Picard group of $C$ 
and $\bbN^{+}$ is the semigroup of strictly positive integers.
\item
We can regard $\M(\tor C)$ and $\M(\vect C)$ as submonoids of $\M(\coh C)$.
Then $\M(\coh C)$ is the disjoint union of 
$\M(\tor C)$ and $\M(\vect C)^+:=\M(\vect C)\setminus \{0\}$.
See Corollary \ref{cor:M coh only mon} for the complete description
of $\M(\coh C)$ as a monoid.
\end{enua}
\end{thma}
See Example \ref{ex:compare M K_0} for the comparison of 
the Grothendieck group $\K_0(\coh C)$ and the Grothendieck monoid $\M(\coh C)$.
As a corollary of this theorem, we have the following:
\begin{thma}[{Corollary \ref{cor:Serre vect}}]\label{thm:C}
Let $C$ be a smooth projective curve over an algebraically closed field $\bbk$.
Then $\vect C$ has no nontrivial Serre subcategories.
\end{thma}


The other direction is a study of the space $\Serre(\catE)$
whose topology is induced by the topology on $\MSpec \M(\catE)$.
We can recover the topology of a noetherian scheme $X$ from the topological space $\Serre(\coh X)$,
and obtain the following result.
\begin{thma}[{Theorem \ref{thm:recover}}]\label{thm:D}
Consider the following conditions for noetherian schemes $X$ and $Y$.
\begin{enua}
\item
$X \iso Y$ as schemes.
\item
$\M(\coh X) \iso \M(\coh Y)$ as monoids.
\item
$X \iso Y$ as topological spaces.
\end{enua}
Then ``$(1)\imply (2) \imply (3)$'' hold.
\end{thma}

The nontrivial part is, of course, the implication ``$(2) \imply (3)$''.
It is surprisingly enough because the Grothendieck monoid $\M(\coh X)$ loses a lot of information
and the Grothendieck group $\K_0(\coh X)$ never recovers the topology of $X$.
The author does not know other algebraic invariants 
which completely recover the topology of schemes.
\begin{ex}\label{ex:counter M K}
Let $\bbk$ be an algebraic closed field.
\begin{enua}
\item
Let $R$ be a finite dimensional commutative $\bbk$-algebra.
Then $\M(\catmod R) \iso \bbN^{\oplus n}$,
where $n$ is the number of maximal ideals of $R$ (see Example \ref{ex:M fd alg}).
In particular, if $R$ is local, then we have $\M(\catmod R) \iso \bbN$.

Consider a finite dimensional commutative local $\bbk$-algebra $R:=\bbk[x,y]/(x^2,xy,y^2)$.
For any $\lmd\in \bbk$,
define an $\bbk$-algebra homomorphism $\phi_{\lmd} \colon R \to M_2(\bbk)$ 
to the matrix algebra $M_2(\bbk)$ of degree $2$ by
\[
\phi_{\lmd}(x)=
\begin{bmatrix}
0& 1 \\
0& 0
\end{bmatrix}
, \quad
\phi_{\lmd}(y)=
\begin{bmatrix}
0& \lmd \\
0& 0
\end{bmatrix}.
\]
Then $\phi_{\lmd}$ defines a $2$-dimensional $R$-module $M_{\lmd}$.
We can easily see that $M_{\lmd}$ is indecomposable
and $M_{\lmd} \not\iso M_{\mu}$ if $\lmd \ne \mu$.
Thus $\catmod R$ has infinitely many indecomposable objects.
On the other hand, $\catmod \bbk$ has exactly one indecomposable object $\bbk$.
Hence $\catmod R$ and $\catmod\bbk$ are very different.
Despite this, by the fact above, we have $\M(\catmod \bbk) \iso \bbN \iso \M(\catmod R)$.
This shows that Grothendieck monoids lose a lot of categorical information.
\item
Let $\bbP^1$ be the projective line over $\bbk$.
Then we have $\K_0(\coh \bbP^1) \iso \bbZ^{\oplus 2}$ (see Example \ref{ex:compare M K_0}).
On the other hand, we have
$\K_0(\coh(\Spec(\bbk \times \bbk)))\iso \K_0(\catmod(\bbk \times \bbk)) \iso \bbZ^{\oplus 2}$.
Hence $\bbP^1 \not\iso \Spec(\bbk \times \bbk)$ as topological spaces,
but $\K_0(\coh \bbP^1) \iso \K_0(\coh(\Spec(\bbk \times \bbk)))$ as groups.
This shows that ``$(2) \imply (3)$'' of Theorem \ref{thm:D} becomes false
if Grothendieck monoids are replaced by Grothendieck groups.
\end{enua}
\end{ex}

Example \ref{ex:counter M K} (1) also shows that
``$(1)\imply (2)$'' in Theorem \ref{thm:D} is strict.
However, the author does not know such an example for ``$(2)\imply (3)$'' at the moment.
\begin{ques}
Is there a pair of noetherian schemes $X$ and $Y$
such that $X \iso Y$ as topological spaces
but $\M(\coh X) \not\iso \M(\coh Y)$ as monoids?
\end{ques}

Finally, 
we describe the relationship between Grothendieck monoids and 
other invariants for noetherian schemes.
It is illustrated as follows:
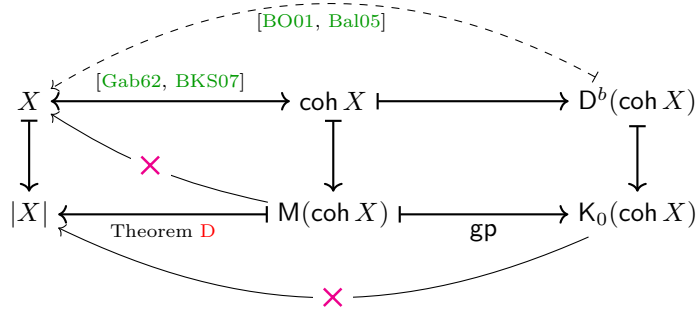
\begin{figure}[H]
\centering
\begin{tikzpicture}
\node (X) at (0,0) {$X$};
\node (coh) at (4,0) {$\coh X$};
\node (top) at (0,-1.5) {$\abs{X}$};
\node (Db) at (8,0) {$\D^{b}(\coh X)$};
\node (M) at (4,-1.5) {$\M(\coh X)$};
\node (K) at (8,-1.5) {$\K_0(\coh X)$};
 \draw[<->,thick] (X) to["{\scriptsize \cite{Gab62,BKS07}}"] (coh);
 \draw[|->,thick] (coh) -- (Db);
 \draw[|->,thick] (coh) -- (M);
 \draw[|->,thick] (Db) -- (K);
 \draw[|->,thick] (M) to["$\gp$"'] (K);
 \draw[|->,thick] (X) -- (top);
 \draw[|->,thick] (M) to["{\scriptsize Theorem \ref{thm:D}}"] (top);
 \draw[|->, dashed] (Db) to[bend right=30,"{\scriptsize \cite{BO01,Balmer}}"] (X);
 \draw[->] (M) to[bend left=10] (X)	;
 \draw[->] (K) to[bend left=25] (top);
 \node[text width=8pt,fill=white] at (4,-2.6) {\large \red{$\Cross$}};
 \node[text width=8pt,fill=white] at (1.6,-0.85) {\large \red{$\Cross$}};
\end{tikzpicture}
\end{figure}
Here
\begin{itemize}
\item 
$\abs{X}$ is the underlying topological space of $X$,
\item
$\D^{b}(\coh X)$ is the bounded derived category of $\coh X$,
\item
$\gp$ is the group completion (see Definition \ref{dfn:gp} and Remark \ref{rmk:compare grp mon}),
\item
the thick arrows $A \bm{\mapsto} B$ indicate that
$B$ can be constructed from $A$,
\item
the arrows $A \to B$ marked with a cross indicate that
$B$ cannot be recovered from $A$ (see Example \ref{ex:counter M K}),
\item
the dashed arrow indicates that some assumption or additional structure is needed.
\end{itemize}

\medskip
\noindent
{\bf Organization.}
This paper is organized as follows.

In Section \ref{s:Serre face}, 
we first review some definitions and properties of commutative monoids,
exact categories and its Grothendieck monoid.
Then we establish a bijection between Serre subcategories of an exact category
and faces of its Grothendieck monoid.

In Section \ref{s:ex fd alg},
we study the Grothendieck monoid of an exact category with finiteness conditions
and classify its Serre subcategories.
In particular, we give an explicit example of classifying Serre subcategories
of a complicated exact category related to a finite dimensional algebra
by using its Grothendieck monoid
(Example \ref{ex:explicit fd}).

In Section \ref{s:ex curve},
we determine the Grothendieck monoids of exact categories related to a smooth projective curve
and classify Serre subcategories of them.

In Section \ref{s:spec},
we first review the spectrum of a monoid and monoidal spaces,
which are natural analogies of the spectrum of a commutative ring
and ringed spaces, respectively.
Next, we introduce a topology on the set of Serre subcategories
and study relations with the spectrum of the Grothendieck monoid.
Finally, we recover the topology of a noetherian scheme $X$ 
from the Grothendieck monoid $\M(\coh X)$.

\medskip
\noindent
{\bf Conventions.}
For a category $\catC$,
we denote by $\abs{\catC}$ the class of all isomorphism classes of objects
and $\Hom_{\catC}(X,Y)$ the set of morphisms between objects $X$ and $Y$ in $\catC$.
The isomorphism class of $X\in \catC$ is also denoted by $X$.
In this paper,
we suppose that all categories are \emph{skeletally small}, that is, 
the class $\abs{\catC}$ forms a set.
Also, we suppose that all subcategories are full subcategories closed under isomorphisms.
We often identify subcategories of $\catC$ and subsets of $\abs{\catC}$.

For a noetherian ring $\LL$ (with identity and associative multiplication),
we denote by $\catmod \LL$ the category of finitely generated (right) $\LL$-modules.
We set $\Hom_{\LL}(X,Y):=\Hom_{\catmod \LL}(X,Y)$.

A \emph{variety} over a field $\bbk$
means a separated integral scheme of finite type over $\bbk$. 
A \emph{curve} is a $1$-dimensional variety.

A \emph{monoid} means a semigroup with a unit.
Every monoid is assumed to be \emph{commutative}.
We use an additive notation, that is,
the operation is denoted by $+$ with its unit $0$.
We denote by $\bbN$ the monoid of non-negative integers.
For a subset $S$ of a monoid $M$,
we denote by $\gen{ S }_\bbN$ the smallest submonoid of $M$ containing $S$.

\medskip
\noindent
{\bf Acknowledgement.}
The author would like to thank Haruhisa Enomoto for valuable comments and discussion.
This work is supported by JSPS KAKENHI Grant Number JP21J21767.

\section{Classifying Serre subcategories via the Grothendieck monoid}\label{s:Serre face}
In this section,
we establish a bijection between the set of Serre subcategories of an exact category
and the set of faces of its Grothendieck monoid (\S \ref{ss:Serre face}).
In the first half of this section,
we review some definitions and properties of commutative monoids (\S \ref{ss:mon}),
exact categories and its Grothendieck monoid (\S \ref{ss:exact}).

Before starting anything,
let us comment on the author's related work \cite{ES}.
\begin{rmk}\label{rmk:ES}
In \cite{ES}, 
the author and Enomoto study the Grothendieck monoid of an \emph{extriangulated category},
which is both a generalization and unification of both triangulated and exact categories.
In the present paper,
we study the Grothendieck monoid by using the notion of \emph{S-equivalence},
whereas, in \cite{ES}, we do that by using the more sophisticated notion of \emph{c-equivalence}.
Most of the content of this section is generalized for extriangulated categories,
and the proofs are simplified by using c-equivalence.
Thus, if the reader is comfortable with extriangulated categories,
this section can be replaced by \cite[Section 2.2 and 3.1]{ES}.
Although this, we leave this paper as written using S-equivalence
since it is intuitive and
there is a slight advantage such as Fact \ref{fct:= ab}.
\end{rmk}

\subsection{Preliminaries: commutative monoids}\label{ss:mon}
We collect minimal definitions and properties of commutative monoids
to describe our first results.
The main reference of this subsection is \cite{Og}.

A \emph{monoid} is a semigroup with a unit.
In this paper, every monoid is assumed to be \emph{commutative}.
Thus we use the additive notation, that is,
the binary operation is denoted by $+$, and the unit is denoted by $0$.
A \emph{monoid homomorphism} is a map $f\colon M \to N$ 
satisfying $f(x+y)=f(x)+f(y)$ and $f(0_M)=0_N$.
We denote $\Mon$ by the category of (commutative) monoids and monoid homomorphisms.
The category $\Mon$ has arbitrary small limits and colimits (see \cite[Section I.1.1]{Og}). 
Note that the forgetful functor $\Mon \to \Set$ preserves small limits,
where $\Set$ denotes the category of sets and maps.
We can define
products $\prod_{i\in I}M_i$ and direct sums (= coproducts) $\bigoplus_{i\in I}M_i$ of monoids
in a similar manner to vector spaces. 
In particular, finite products and finite direct sums coincide.
A subset $N$ of a monoid $M$ is called a \emph{submonoid}
if it is closed under the multiplication of $M$ and contains the identity element $0$ of $M$.

A basic example of monoids is
the set $\bbN$ of non-positive integers with the arithmetic addition.
A monoid $M$ is said to be \emph{free} 
if it is isomorphic to $\bbN^{\oplus I}$ for some index set $I$.
The \emph{rank} and a \emph{basis} of a free monoid $M$
are defined in similar ways to vector spaces.

\begin{rmk}
For a monoid homomorphism $f\colon M \to N$,
define a submonoid of $M$ by
\[
\Ker (f) :=
\{x\in M \mid f(x)=0  \}.
\]
A caution is that \emph{the condition $\Ker (f)=0$ does not imply $f$ is injective}.
Indeed, the monoid homomorphism
\[
f\colon \bbN^{\oplus 2} \to \bbN,\quad
(x,y) \mapsto x+y
\]
is not injective but $\Ker (f)=0$.
\end{rmk}

The notion of quotients of monoids slightly differs from that of vector spaces.
%
Let $M$ be a monoid and $N$ its submonoid.
Define an equivalence relation on $M$ by
\[
x\sim_N y :\equi \text{there exist $n,n'\in N$ such that $x+n=y+n'$.}
\]
Then the quotient set $M/N:=M/{\sim_N}$ has a natural monoid structure given by
\[
(x \bmod N) + (y \bmod N):=(x+y) \bmod N,
\]
where $x \bmod N$ denotes the equivalence class containing $x\in M$.
This monoid $M/N$ is called the \emph{quotient monoid} of $M$ by $N$.
The quotient map $M\to M/N$ is a monoid homomorphism.
We often write $x\equiv y \bmod N$ if $x \sim_N y$ for $x,y\in M$.
We can easily see that the quotient monoids have the following universal property.
\begin{prp}\label{prp:mon quot}
Let $N$ be a submonoid of a monoid $M$ and
let $\pi \colon M\to M/N$ be the quotient homomorphism.
Then for any monoid homomorphism $f\colon M \to X$ such that $f(N)=0$,
there exists a unique monoid homomorphism $\ol{f}\colon M/N \to X$ satisfying $f=\ol{f}\pi$.
This means that the diagram
\[
N \underset{0}{\overset{\iota}{\rightrightarrows}} M \xr{\pi} M/N
\]
is a coequalizer diagram in $\Mon$,
where $\iota$ is the inclusion map.
\end{prp}

%
Next, we introduce a class of submonoids which corresponds to
Serre subcategories in \S \ref{s:Serre face}.
\begin{dfn}\label{dfn:face}
Let $M$ be a monoid.
\begin{enua}
\item
A submonoid $F$ of $M$ is called a \emph{face}
if for all $x,y\in M$, we have that $x+y\in F$ if and only if both $x\in F$ and $y \in F$.
\item
$\Face(M)$ denotes the set of faces of $M$.
\end{enua}
\end{dfn}

We regard $\Face(M)$ as a poset by the inclusion-order.

\begin{rmk}\label{rmk:Face}
Let $M$ be a monoid.
An element $x\in M$ is a \emph{unit} if there exists $y\in M$ such that $x+y=0$.
We denote by $M^{\times}$ the set of units of $M$.
Then $M^{\times}$ is the smallest face.
On the other hand, $M$ itself is the largest face of $M$.
Thus $\Face(M)$ has the maximum and minimum elements.
$\Face(M)$ consists of exactly one point if and only if $M$ is a group.
\end{rmk}

\begin{ex}
Let $M$ be a free monoid of rank $2$ with a basis $e_1$ and $e_2$.
\begin{enua}
\item
$\bbN(e_1+e_2)$ is a submonoid of $M$ but not a face.
\item
A face of $M$ is one of the following:
$M$ itself, $\bbN e_1$, $\bbN e_2$ or $0$.
\end{enua}
\end{ex}

We list the properties of monoids which we will use.
\begin{dfn}\label{dfn:property mon}
Let $M$ be a monoid.
\begin{enua}
\item
$M$ is \emph{sharp} (or \emph{reduced}) 
if $a+b=0$ implies $a=b=0$ for any $a,b\in M$.
\item
$M$ is \emph{cancellative} (or \emph{integral}) 
if $a + x = a + y$ implies $x = y$ for any $a, x, y \in M$.
\end{enua}
\end{dfn}

Finally, we discuss the relationship between monoids and groups.
\begin{dfn}\label{dfn:gp}
The \emph{group completion} of a monoid $M$
is a pair $(\gp M, \rho)$ of a group $\gp M$ and a monoid homomorphism $\rho \colon M \to \gp M$
satisfying the following universal property:
\begin{itemize}
\item 
For every monoid homomorphism $f\colon M \to G$ into a group $G$,
there exists a unique group homomorphism $\ol{f}\colon \gp M \to G$
such that $f=\ol{f}\rho$.
\end{itemize}
\end{dfn}
The group completion $\gp M$ actually exists for any monoid $M$.
It is constructed as the localization of $M$ with respect to $M$ itself
(see Definition \ref{def:mon-loc}).
The group completion has the following properties by the construction.
\begin{prp}[{cf. Definition \ref{def:mon-loc}}]\label{prp:property gp}
Let $M$ be a monoid and $(\gp M, \rho \colon M \to \gp M)$ its group completion.
\begin{enua}
\item
$\gp M$ is an abelian group.
\item
For any $x,y\in M$,
the equality $\rho(x)=\rho(y)$ holds in $\gp M$ if and only if
$x+s = y+s$ in $M$ for some $s\in M$.
\end{enua}
\end{prp}

The cancellation property is related to the group completion as follows.
\begin{prp}\label{prp:grp comp}
Let $M$ be a monoid and $(\gp M, \rho \colon M \to \gp M)$ its group completion.
Then the following are equivalent.
\begin{enua}
\item
$M$ is cancellative.
\item
The monoid homomorphism $\rho \colon M \to \gp M$
is injective.
\item
There is an injective monoid homomorphism from $M$ to some group.
\end{enua}
\end{prp}
\begin{proof}
It follows immediately from Proposition \ref{prp:property gp} 
and the fact that any submonoid of an abelian group is cancellative.
\end{proof}

\subsection{Preliminaries: Grothendieck monoids of exact categories}\label{ss:exact}
In this subsection,
we first give a brief review of an exact category
and then define the Grothendieck monoid of an exact category,
which is the main subject of this paper.

Let $\catA$ be an abelian category.
An additive subcategory $\catE \subseteq \catA$ is said to be \emph{extension-closed} if
for any exact sequence $0\to X \to Y \to Z \to 0$ in $\catA$,
the condition $X, Z\in \catE$ implies $Y \in \catE$.
We also say that $\catE$ is an \emph{exact category}
when we want to omit the ambient abelian category $\catA$.
A \emph{conflation} in $\catE$ is
an exact sequence $0\to X \to Y \to Z \to 0$ in $\catA$
with $X,Z \in \catE$.
Note that $Y$ also belongs to $\catE$ automatically.
For a conflation $0\to X \xr{f} Y \xr{g} Z \to 0$ in $\catE$,
the morphism $f$ (resp. $g$) is called an \emph{inflation} (resp. \emph{deflation}).
In this case, we say that $Z$ (resp. $X$) is the cokernel of the inflation $f$ 
(resp. the kernel of the deflation of $g$).
An inflation (resp. a deflation) is sometimes denoted by $X \infl Y$ (resp. $Y \defl Z$).
\begin{ex}
Consider the category $\catmod \bbZ$ of finitely generated modules over
the ring $\bbZ$ of integers.
The category $\proj \bbZ$ of finitely generated free (=projective) $\bbZ$-modules 
is an extension-closed subcategory of the abelian category $\catmod \bbZ$.
Then the sequence
\[
0 \to \bbZ \xr{\begin{bsmatrix}1 \\ 1 \end{bsmatrix}} \bbZ^{\oplus 2}
\xr{\begin{bsmatrix}1 & -1 \end{bsmatrix}} \bbZ \to 0
\]
is a conflation in $\proj \bbZ$,
while the sequence
\[
0 \to \bbZ \xr{2} \bbZ \to \bbZ/2\bbZ \to 0
\]
is not a conflation in $\proj \bbZ$.
Thus the monomorphism $\bbZ \xr{\begin{bsmatrix}1 \\ 1 \end{bsmatrix}} \bbZ^{\oplus 2}$
is an inflation in $\proj \bbZ$
but the monomorphism $\bbZ \xr{2} \bbZ$ is not.
\end{ex}
Let $\catE$ be an exact category.
An additive subcategory $\catF \subseteq \catE$ is said to be \emph{conflation-closed} if
for any conflation $0\to X \to Y \to Z \to 0$ in $\catE$,
the condition $X, Z\in \catF$ implies $Y \in \catF$.
For an abelian category,
conflation-closed subcategories coincide with extension-closed subcategories.
A conflation-closed subcategory $\catF \subseteq \catE$ is also an exact category
which has the same ambient abelian category as that of $\catE$.

\begin{rmk}
There is an axiomatic definition of exact categories 
which does not depend on the embedding $\catE \inj \catA$ to an abelian category.
See \cite{Bu10} for sophisticated and standard treatments of exact categories.
We do not follow this formulation for accessibility.
Our treatment of exact categories is justified by Thomason's embedding theorem
(cf. \cite[Appendix A]{Bu10}).
Even if the reader thinks of exact categories as the axiomatic ones,
there is no problem at all in this paper.
\end{rmk}

Let us introduce the Grothendieck monoid $\M(\catE)$ of an exact category $\catE$,
a natural monoid version of the Grothendieck group $\K_0(\catE)$.
It is defined by some universal property.

\begin{dfn}[{\cite[Definition 3.2]{En22}}]\label{dfn:Gro mon UMP}
Let $\catE$ be an exact category.
\begin{enua}
\item
An \emph{additive function} on $\catE$ with values in a monoid $M$
is a map $f\colon \abs{\catE} \to M$ satisfying the following conditions:
\begin{enur}
\item
$f(0)=0$ holds.
\item
For any conflation $0\to X \to Y\to Z \to 0$ in $\catE$,
we have that $f(Y)=f(X)+f(Z)$ in $M$.
\end{enur}
\item
A \emph{Grothendieck monoid} $\M(\catE)$ of $\catE$
is a monoid $\M(\catE)$ together with an additive function $\pi\colon \abs{\catE} \to \M(\catE)$ 
which satisfies the following universal property:
\begin{itemize}
\item
For any additive function $f\colon \abs{\catE} \to M$ with values in a monoid $M$,
there exists a unique monoid homomorphism $\ol{f}\colon \M(\catE) \to M$
such that $f=\ol{f}\pi$.
\[
\begin{tikzcd}
{\abs{\catE}} \arr{r,"f"} \arr{d,"\pi"'} & M\\
\M(\catE) \arr{ru,"\ol{f}"'} & .
\end{tikzcd}
\]
\end{itemize}
We often write $[X]:=\pi(X)$ for $X\in \abs{\catE}$.
\end{enua}
\end{dfn}

The Grothendieck monoid $\M(\catE)$ actually exists for any exact category $\catE$
(see \cite[Proposition 3.3]{En22}).
\begin{ex}
Let $\catE$ be an exact category.
The following equalities hold in $\M(\catE)$.
\begin{itemize}
\item
$[X]+[Z]=[Y]$ for any conflation $0\to X\to Y\to Z\to0$.
It follows from the assignment $X\mapsto [X]$ is an additive function.
\item
$[X]+[Y]=[X\oplus Y]$ for any objects $X, Y\in\catE$.
Indeed, there is a split conflation $0\to X \to X\oplus Y \to Y \to 0$.
\end{itemize}
\end{ex}

We introduce some terminologies 
to give a direct characterization of whether $[X]=[Y]$ in $\M(\catE)$.
Let $X$ be an object of an exact category $\catE$.
Two inflations $Y\infl X$ and $Z \infl X$ are \emph{equivalent}
if there is an isomorphism $Y\isoto Z$ such that the following diagram commutes:
\[
\begin{tikzcd}[row sep=3]
Y \arr{rd,rightarrowtail} \arr{dd,"\iso"'} &\\
 & X \\
Z \arr{ru,rightarrowtail} &.
\end{tikzcd}
\]
An \emph{admissible subobject} of $X$ is 
the equivalence class of an inflation $Y\infl X$.
We often say that $Y$ is an admissible subobject of $X$ 
and denote the cokernel of $Y\infl X$ by $X/Y$.
We omit the adjective \emph{admissible} if $\catE$ is an abelian category.
For two admissible subobjects $Y$ and $Z$ of $X$,
we write $Y \le Z$ if there exists an inflation $Y \infl Z$ 
such that the following diagram commutes:
\[
\begin{tikzcd}[row sep=3]
Y \arr{rd,rightarrowtail} \arr{dd,rightarrowtail} &\\
 & X \\
Z \arr{ru,rightarrowtail} &.
\end{tikzcd}
\]
This binary relation $\le$ yields a partial order on the set of admissible subobjects of $X$.
See \cite[Section 2]{En22} for a detailed study of the poset of admissible subobjects.

An \emph{admissible subobject series} of $X$ is
a finite sequence $0 = X_0 \le X_1 \le \cdots \le X_n=X$ of admissible subobjects of $X$.
Two admissible subobject series $0 = X_0 \le X_1 \le \cdots \le X_n=X$ and
$0 = Y_0 \le Y_1 \le \cdots \le Y_m=Y$ are \emph{isomorphic}
if $n=m$ and there exists a permutation $\sg\in\symg_n$
such that $X_i/X_{i-1} \iso Y_{\sg(i)}/Y_{\sg(i)-1}$ for all $1\le i \le n$.
In this case, we say that $X$ and $Y$ are \emph{S-equivalent}
\footnote{The terminology \emph{S-equivalence} comes from \cite{Ki94},
which was originally introduced in \cite{Se67}.
The original notion of S-equivalence is used 
to study the moduli space of representations of a finite dimensional algebra.}
and denote it by $X \sim_S Y$.
If $X \sim_S Y$ holds, then we have the following equality in $\M(\catE)$:
\[
[X]=[X_1]+[X_2/X_1]+\cdots +[X_n/X_{n-1}]=
[Y_1]+[Y_2/Y_1]+\cdots +[Y_n/Y_{n-1}]=[Y].
\]
Conversely,
the following holds.
\begin{fct}[{\cite[Proposition 3.4]{En22}}]\label{fct:=}
Let $\catE$ be an exact category.
For any two objects $X,Y \in \catE$,
the equality $[X]=[Y]$ holds in $\M(\catE)$ if and only if
there exists a sequence of objects $X=X_0, X_1, \dots, X_m=Y$ in $\catE$
such that $X_{i-1} \sim_S X_{i}$ for each $i$.
\end{fct}

If $\catE$ is an abelian category, we can strengthen this fact.
\begin{fct}[{\cite[Proposition 3.3]{Br98}}]\label{fct:= ab}
Let $\catA$ be an abelian category.
For any two objects $X,Y\in\catA$,
the equality $[X]=[Y]$ holds in $\M(\catA)$ if and only if 
$X$ and $Y$ are S-equivalent.
\end{fct}

Using this description, we can obtain the following properties of the Grothendieck monoid.
\begin{fct}[{\cite[Proposition 3.5]{En22}}]\label{fct:sharp}
Let $\catE$ be an exact category.
\begin{enua}
\item
For an object $X \in \catE$,
we have that $[X]=0$ in $\M(\catE)$ if and only if $X\iso 0$.
\item
$\M(\catE)$ is sharp (see Definition \ref{dfn:property mon}).
\end{enua}
\end{fct}

Next, we investigate a functorial property of Grothendieck monoids.
An additive functor $F\colon \catD \to \catE$ between exact categories
is said to be \emph{exact} if for any conflation $0\to X\to Y\to Z \to 0$ in $\catD$,
the sequence $0\to FX\to FY\to FZ \to 0$ is also a conflation in $\catE$.
An \emph{exact equivalence} is an exact functor which has an exact quasi-inverse.
Note that a fully faithful and essentially surjective exact functor
is not necessarily an exact equivalence.

Let $F\colon \catD \to \catE$ be an exact functor between exact categories.
Define an additive function $\abs{\catD} \to \M(\catE)$ by $X \mapsto [FX]$.
It gives rise to a monoid homomorphism $\M(F) \colon \M(\catD) \to \M(\catE)$ 
by the universal property of $\M(\catD)$.
For two exact functors $F\colon \catE_1 \to \catE_2$ and $G\colon \catE_2 \to \catE_3$,
we have $\M(GF)=\M(G)\M(F)$.
It is easy to check that for any exact equivalence $F\colon \catD \simto \catE$,
the monoid homomorphism $\M(F) \colon \M(\catD) \to \M(\catE)$ is an isomorphism.

Finally, we compare the Grothendieck monoid $\M(\catE)$ with the Grothendieck group $\K_0(\catE)$.
\begin{rmk}\label{rmk:compare grp mon}
Let $\catE$ be an exact category.
\begin{enua}
\item
Recall that the \emph{Grothendieck group} $\K_0(\catE)$ of $\catE$ is defined by
\[
\K_0(\catE):=\bigoplus_{X\in \abs{\catE}}\bbZ X\Big{/}\la A-B+C 
\mid \text{$0\to A \to B\to C \to 0$ is a conflation} \ra.
\]
The image of $X\in\abs{\catE}$ in $\K_0(\catE)$ is denoted by $[X]$.
Then there is a natural monoid homomorphism
\[
\rho \colon \M(\catE) \to \K_0(\catE),\quad [X] \mapsto [X].
\]
We can easily check that 
$(\K_0(\catE), \rho \colon \M(\catE) \to \K_0(\catE))$ is the group completion
(see Definition \ref{dfn:gp}).
\item
The natural map $\rho$ is injective if and only if
$\M(\catE)$ is cancellative by Proposition \ref{prp:grp comp}.
In this case,
the Grothendieck monoid $\M(\catE)$ can be identified with
the positive part
\[
\K_0^+(\catE):=\{[X]\in \K_0(\catE) \mid X\in \catE\}
\]
of the Grothendieck group.
Thus if $\M(\catE)$ is cancellative,
the computation of $\M(\catE)$ becomes much easier.
However,
not much is known about the conditions for an exact category $\catE$
under which $\M(\catE)$ becomes cancellative.
Nevertheless, we prove that the Grothendieck monoid $\M(\vect C)$ 
of the category of vector bundles over a smooth projective curve
is cancellative in Proposition \ref{prp:M vect}.
\item
An element of $\M(\catE)$ can be expressed by $[X]$
for some single object $X\in \catE$,
while an element of $\K_0(\catE)$ can only be expressed by $[X]-[Y]$
for some objects $X, Y \in \catE$ in general.
It is an advantage of the Grothendieck monoid.
\item
Grothendieck monoids are more rigid than Grothendieck groups
in the sense that they are 
not invariant under triangulated equivalences of the bounded derived categories $\D^b(\catE)$
but only invariant under exact equivalences of exact categories
(see Example \ref{ex:compare M K_0}).
\end{enua}
\end{rmk}

\subsection{Serre subcategories and faces}\label{ss:Serre face}
Throughout this subsection, $\catE$ is an exact category.
Let us give a correspondence between subcategories of $\catE$ and subsets of $\M(\catE)$.
For a subcategory $\catD$ of $\catE$,
we define a subset of $\M(\catE)$ by
\[
\M_{\catD}:=\{[X] \in \M(\catE) \mid X\in \catD \}.
\]
For a subset $N\subseteq \M(\catE)$,
we define a subcategory of $\catE$ by
\[
\catD_{N}:=\{X\in \catE \mid [X] \in N\}.
\]
We will show that these assignments give a bijection between certain subcategories of $\catE$
and certain subsets of $\M(\catE)$.

\begin{dfn}
A subcategory $\catD$ of $\catE$ is said to be \emph{closed under S-equivalences} if
$X\sim_S Y$ and $X\in\catD$ implies $Y\in \catD$ for any $X,Y \in\catE$.
In this case, we also say that $\catD$ is an \emph{S-closed subcategory} for short.
\end{dfn}

The relation between the categories $\catD_N$ defined above 
and S-closed subcategories is as follows.
\begin{prp}\label{prp:S-closed bij}
The following hold.
\begin{enua}
\item
For a subset $N\subseteq \M(\catE)$,
the subcategory $\catD_{N}$ is closed under S-equivalences.
\item
The maps $\Phi \colon \catD \mapsto \M_{\catD}$ and $\Psi \colon N\mapsto \catD_N$ are 
mutually inverse bijections between
the set of $S$-closed subcategories and the power set of $\M(\catE)$.
\end{enua}
\end{prp}
\begin{proof}
(1)
Let $X\in\catD_N$ and $Y\in \catE$ such that $X \sim_S Y$.
Then we have $[Y]=[X]\in N$,
and hence $Y$ also belongs to $\catD_N$,
which proves $\catD_N$ is S-closed.

(2)
We first prove that $\M_{\catD_N}=N$.
Clearly, we have $\M_{\catD_N}\supseteq N$.
Take $[X] \in \M_{\catD_N}$.
Then there exists $Y\in \catD_N$ such that $[X]=[Y]$.
Because $\catD_N$ is closed under S-equivalences by (1),
the object $X$ also belongs to $\catD_N$, and thus $[X]\in N$,
which shows that $\M_{\catD_N}=N$.
Next, we prove that $\catD_{\M_{\catD}}=\catD$ 
for an S-closed subcategory $\catD$ of $\catE$.
It is obvious that $\catD_{\M_{\catD}} \supseteq \catD$.
Take $X\in \catD_{\M_{\catD}}$.
We have $[X]\in \M_\catD$,
and then there exists $Y \in \catD$ such that $[X]=[Y]$.
Since $\catD$ is closed under S-equivalences, we obtain that $X\in\catD$,
which proves $\catD_{\M_{\catD}}=\catD$.
Therefore $\Phi$ and $\Psi$ are mutually inverse bijections.
\end{proof}

In what follows, we translate the properties of subcategories of $\catE$
into those of subsets of the Grothendieck monoid $\M(\catE)$.
We recall that properties of a subcategory of an exact category $\catE$.
\begin{itemize}
\item 
An additive subcategory $\catD$ of $\catE$ is said to be \emph{closed under direct summands}
if $X\oplus Y \in\catD$ implies that both $X$ and $Y$ belong to $\catD$
for any objects $X,Y\in \catE$.
\item 
An additive subcategory $\catS$ of $\catE$ is called a \emph{Serre subcategory}
if for any conflation $0\to X \to Y \to Z \to 0$ in $\catE$,
we have that $X,Z\in\catS$ if and only if $Y\in\catS$.
\end{itemize}
Serre subcategories are clearly conflation-closed.
Thus they are also exact categories.

\begin{lem}[{cf. \cite[Proposition 3.7]{En22}}]\label{lem:Serre S-closed}
A Serre subcategory $\catS$ of $\catE$ is closed under direct summands and S-equivalences.
\end{lem}
\begin{proof}
Let $\catS$ be a Serre subcategory of $\catE$.
Then $\catS$ is closed under direct summands
since there is a conflation $0\to X \to X\oplus Y\to Y \to 0$ for any $X,Y\in\catE$.
We prove that $\catS$ is closed under S-equivalences.
Let $X\in \catS$ and $Y\in \catE$ satisfying $X \sim_S Y$.
There are admissible subobject series 
$0 = X_0 \le X_1 \le \cdots \le X_n=X$ and $0 = Y_0 \le Y_1 \le \cdots \le Y_n=Y$,
and a permutation $\sg\in\symg_n$ such that $X_i/X_{i-1} \iso Y_{\sg(i)}/Y_{\sg(i)-1}$ for all $i$.
Since $\catS$ is Serre and $X\in\catS$,
we have that $X_i/X_{i-1}\in \catS$,
which implies $Y_i/Y_{i-1}\in \catS$ for all $i$.
Thus we conclude that $Y$ belongs to $\catS$ since it is conflation-closed.
\end{proof}

\begin{lem}\label{lem:sub M}
The following hold.
\begin{enua}
\item
If $\catD$ is an additive subcategory of $\catE$,
then $\M_{\catD}$ is a submonoid of $\M(\catE)$.
\item
If $N$ is a submonoid of $\M(\catE)$,
then $\catD_N$ is a conflation-closed subcategory of $\catE$.
\item
If $\catS$ is a Serre subcategory,
then $\M_{\catS}$ is a face of $\M(\catE)$ (see Definition \ref{dfn:face}).
\item
If $F$ is a face of $\M(\catE)$,
then $\catD_F$ is a Serre subcategory of $\catE$.\end{enua}
\end{lem}
\begin{proof}
(1), (2) and (4) follow immediately.
We only prove (3).
Let $\catS$ be a Serre subcategory of $\catE$.
Then $\M_{\catS}$ is a submonoid of $\M(\catE)$ by (1).
Suppose that $[X]+[Y] \in \M_{\catS}$ for some objects $X,Y\in\catE$.
Then there exists $Z\in \catS$ such that $[Z]=[X]+[Y]=[X\oplus Y]$.
This yields $X\oplus Y \in\catS$ because $\catS$ is S-closed.
Since $\catS$ is closed under direct summands,
both $X$ and $Y$ belong to $\catS$,
and hence we have that both $[X] \in \M_{\catS}$ and $[Y]\in\M_{\catS}$.
This proves $\M_\catS$ is a face of $\M(\catE)$.
\end{proof}


\begin{cor}\label{cor:Serre bij}
The bijections in Proposition \ref{prp:S-closed bij} restricts to
inclusion-preserving bijections between the following sets:
\begin{itemize}
\item
The set $\Serre(\catE)$ of Serre subcategories of $\catE$.
\item
the set $\Face(\M(\catE))$ of faces of $\M(\catE)$.
\end{itemize}
\end{cor}
\begin{proof}
It follows from Proposition \ref{prp:S-closed bij}, Lemma \ref{lem:Serre S-closed} and Lemma \ref{lem:sub M}.
\end{proof}

We think of classifying Serre categories of an exact category
as a special case of classifying faces of a monoid
by this corollary.
Following this philosophy, we classify faces of an abstract monoid
and apply it to classify Serre subcategories in \S \ref{s:ex fd alg} and \ref{s:ex curve}.

Finally, we compare $\M_{\catS}$ with $\M(\catS)$ for a Serre subcategory $\catS$.
\begin{prp}\label{prp:inj Serre}
Let $\catS$ be a Serre subcategory of $\catE$,
and let $\iota \colon \catS \inj\catE$ be the natural inclusion functor.
Then $\M(\iota) \colon \M(\catS) \to \M(\catE)$ is an injective monoid homomorphism
whose image is $\M_{\catS} \subseteq \M(\catE)$.
\end{prp}
\begin{proof}
It is clear that the image of $M(\iota)$ is $\M_{\catS}$.
Thus we only show that $\M(\iota)$ is injective.
By Lemma \ref{lem:Serre S-closed} and Fact \ref{fct:=},
it is enough to show that
for any $X,Y \in\catS$,
if $X$ and $Y$ are S-equivalent in $\catE$,
then they are also S-equivalent in $\catS$.
The same consideration as in Lemma \ref{lem:Serre S-closed} works to prove this.
\end{proof}
For a Serre subcategory $\catS$ of $\catE$,
we often identify $\M(\catS)$ with the face $\M_{\catS}$ of $\M(\catE)$
by this proposition.

\begin{rmk}\label{rmk:Brookfield}
Corollary \ref{cor:Serre bij} and Proposition \ref{prp:inj Serre} are not entirely new.
Those are originally mentioned in \cite[Proposition 16.8]{Br97} for the category $\Mod \LL$
of \emph{all} modules over a (\emph{not necessarily noetherian}) ring $\LL$.
Brookfield defined the Grothendieck monoids $\M(\catS)$ 
for Serre subcategories $\catS$ of $\Mod \LL$.
He studied mainly the case $\catS=\noeth \LL$,
the category of noetherian $\LL$-modules,
and used the bijection to identify $\M(\noeth \LL)$
with the face $\M_{\noeth \LL} \subseteq \M(\Mod \LL)$ in our terminologies.
Our approach using the notion of S-closed subcategories is quite different from Brookfield's one
and has a broader application for classifying certain subcategories.
See, for example, \cite[Section 3.2]{ES}.
\end{rmk}

\section{The case of exact categories related to finite dimensional algebras}\label{s:ex fd alg}
In this section,
we give concrete examples of classifying Serre subcategories of an exact category $\catE$ 
by using its Grothendieck monoid $\M(\catE)$.
Our strategy is the following:
\begin{enua}
\item
Relate the Grothendieck monoid $\M(\catE)$ with an abstract monoid $M$.
\item
Classify faces of the abstract monoid $M$.
\item
Classify Serre subcategories of $\catE$ by using (1) and (2).
\end{enua}
Although we think that some results in this section are well-known to experts,
we give the proofs from the viewpoint of Grothendieck monoids.

In \S \ref{ss:more on face}, 
we classify faces of some abstract monoid generated by a subset.
In \S \ref{ss:length exact}, 
we classify Serre subcategories of an exact category with finiteness conditions 
by using the result of \S \ref{ss:more on face}.
In particular, we give an explicit example of classifying Serre subcategories
of a complicated exact category related to a finite dimensional algebra
(Example \ref{ex:explicit fd}).
\subsection{More properties on faces}\label{ss:more on face}
Hereafter $M$ is a monoid.
We first give a description of the face generated by a subset.
This explicit description is useful to study faces.
\begin{fct}[{\cite[Proposition I.1.4.2]{Og}}]\label{fct:description face}
Let $S$ be a subset of $M$.
\begin{enua}
\item
The smallest submonoid of $M$ containing $S$
is equal to
\[
\gen{S}_{\bbN}
:=\gen{x \mid x\in S}_{\bbN}
:=\left\{\sum_{i=1}^m n_ix_i \middle|\; m,n_i \in \bbN, x_i\in S  \right\}.
\]
We call it the \emph{submonoid of $M$ generated by $S$}.
\item
The smallest face of $M$ containing $S$
is equal to
\[
\la S \ra_{\face}
:=\gen{x \mid x\in S}_{\face}
:=\left\{x\in M \mid
\text{there exists $y\in M$ such that $x+y\in \la S \ra_{\bbN}$} \right\}.
\]
We call it the \emph{face of $M$ generated by $S$}.
\end{enua}
\end{fct}

We study the relationship between faces of $M$ and those of its quotient monoid $M/N$.
Unlike the case of vector spaces,
the submonoids of $M/N$ do not correspond to
the submonoids of $M$ containing $N$.
\begin{ex}
Let $M:=\bbN^{\oplus 2}$ and $N:=\bbN(1,0)+\bbN(1,1) \subseteq M$.
Then we have $M/N = 0$ but $M$ and $N$ themselves are distinct submonoids of $M$ containing $N$.
\end{ex}

However, we have a bijection for faces.
\begin{prp}\label{prp:face quot}
Let $N$ be a submonoid of $M$,
and let $\pi \colon M \to M/N$ be the quotient homomorphism.
\begin{enua}
\item
If $F$ is a face of $M$ containing $N$,
then $F/N:=\pi(F)$ is also a face of $M/N$.
\item
If $F'$ is a face of $M/N$,
then $\pi^{-1}(F')$ is also a face of $M$ containing $N$.
\item
The assignments given in (1) and (2) give inclusion-preserving bijections
between the set of faces of $M$ containing $N$ and that of $M/N$.
\end{enua}
\end{prp}
\begin{proof}
We only prove (1) and (3) since the proof of (2) is straightforward.

(1)
It is clear that $F/N$ is a submonoid of $M/N$.
Let $a,b\in M$ such that $\pi(a)+\pi(b) \in F/N$.
Then there exist $x\in F$ and $n, n'\in N$ such that $x+n=(a+b)+n'$ in $M$.
Since $x+n \in F$ and $F$ is a face of $M$, we have that $a,b \in F$.
Thus both $\pi(a)$ and $\pi(b)$ belong to $F/N$,
which proves $F/N$ is a face.

(3)
We have that $\pi^{-1}(F')/N=\pi(\pi^{-1}(F'))=F'$ since $\pi$ is surjective.
It is easy to check that $\pi^{-1}(F/N)\supseteq F$.
It remains to show that $\pi^{-1}(F/N) \subseteq F$.
Let $a \in \pi^{-1}(F/N)$.
Then we have $\pi(a)\in F/N$.
There exist $x\in F$ and $n, n'\in N$ such that $x+n=a+n'$.
Since $x+n \in F$ and $F$ is a face, we have $a\in F$,
which proves $\pi^{-1}(F/N) \subseteq F$.
\end{proof}

\begin{cor}\label{cor:face unit}
For a submonoid $N$ of $M$,
there is an inclusion-preserving bijection between $\Face(M/N)$ and $\Face(M/\la N \ra_{\face})$.
In particular,
we have an inclusion-preserving bijection $\Face(M) \iso \Face(M/M^{\times})$,
where $M^{\times}$ is the set of units of $M$.
\end{cor}
\begin{proof}
It follows from Remark \ref{rmk:Face} and Proposition \ref{prp:face quot}.
\end{proof}

Let $f\colon M\to N$ be a monoid homomorphism.
For any face $F$ of $N$,
the inverse image $f^{-1}(F)$ is also a face of $M$.
Thus we have an inclusion-preserving map $\Face(f)\colon \Face(N) \to \Face(M)$.
The following lemma is obvious but useful.
\begin{lem}\label{lem:surj induce inj on face}
The map $\Face(f)\colon \Face(N) \to \Face(M)$ is injective
for a surjective monoid homomorphism $f\colon M \to N$.
\end{lem}
\begin{proof}
It is straightforward.
\end{proof}


Let us consider a finiteness condition on a monoid
and classify faces of a monoid satisfying it.
\begin{dfn}\hfill
\begin{enua}
\item
$M$ is \emph{finitely generated}
if $M=\gen{S}_{\bbN}$ for a finite subset $S$ of $M$.
\item
A face $F$ of $M$ is \emph{finitely generated}
if $F=\gen{S}_{\face}$ for a finite subset $S$ of $F$.
\end{enua}
\end{dfn}

\begin{rmk}\hfill
\begin{enua}
\item
If $M$ is finitely generated, then it is finitely generated as a face.
\item
A face $F$ of $M$ is finitely generated if and only if
$F=\gen{x}_{\face}$ for some element $x\in M$.
Indeed, if $F=\gen{S}_{\face}$ for a finite subset $S$ of $M$,
then we can easily see that $F=\gen{\sum_{s\in S} s}_{\face}$.
\end{enua}
\end{rmk}

\begin{lem}\label{lem:gen face}
If $M$ is generated by a (not necessarily finite) subset $S\subseteq M$,
then the map
\[
\gen{-}_{\face} \colon \pow(S) \to \Face(M),\quad 
A \mapsto \gen{A}_{\face}
\]
is an inclusion-preserving surjection,
where $\pow(S)$ is the power set of $S$.
\end{lem}
\begin{proof}
Let $F$ be a face of $M$ and set $S_F:=\{x\in S \mid x \in F \}$.
We want to show that $F=\gen{S_F}_{\face}$.
We may assume that $F \ne 0$.
It is clear that $F\supseteq \gen{S_F}_{\face}$.
Take $0\ne x\in F$.
Then $x = \sum_{i=1}^m n_i s_i$ for some $s_i\in S$ and $0 \ne n_i\in \bbN$
by Fact \ref{fct:description face}.
Since $F$ is a face, we obtain that $s_i \in F$ for all $i$,
which shows $x\in \gen{S_F}_{\face}$.
Thus we conclude that $F=\gen{S_F}_{\face}$ and 
$\gen{-}_{\face} \colon \pow(S) \to \Face(M)$ is surjective.
Note that we actually prove $F=\gen{S_F}_{\bbN}$.
\end{proof}

\begin{cor}
If $M$ is finitely generated, then $\Face(M)$ is a finite set.
\end{cor}
\begin{proof}
There is a finite subset $S$ of $M$ such that $M=\gen{S}_{\bbN}$
because $M$ is finitely generated.
Then $\pow(S)$ is also a finite set,
and we conclude that $\Face(M)$ is a finite set by Lemma \ref{lem:gen face}.
\end{proof}

\begin{ex}\label{ex:face free monoid}
Let $M$ be a free monoid with a basis $\{e_i\mid i\in I\}$.
Then it is clear that 
the map $\gen{-}_{\face} \colon \pow(\{e_i\mid i\in I\}) \to \Face(M)$ is bijective.
\end{ex}

\subsection{Serre subcategories of lenght exact categories}\label{ss:length exact}
In this subsection,
we study the Grothendieck monoid of an exact category with finiteness conditions
and apply it to classify Serre subcategories of exact categories 
related to finite dimensional algebras.

We quickly review terminologies related to composition series
to introduce finiteness conditions of exact categories.
Let $\catE$ be an exact category.
An admissible subobject series $0 = X_0 \le X_1 \le \cdots \le X_n=X$ of $X\in \catE$
is \emph{proper} if $X_{i+1}/X_{i}\ne 0$ for all $i$.
In this case, we say that this proper admissible subobject series has \emph{length $n$}.
A nonzero object $X \in \catE$ is said to be \emph{simple}
if it has no admissible subobject except $0$ and $X$ itself.
We denote by $\simp \catE$ the set of isomorphism classes of simple objects of $\catE$.
An admissible subobject series $0 = X_0 \le X_1 \le \cdots \le X_n=X$ of $X\in \catE$
is a \emph{composition series} if $X_{i+1}/X_{i}$ is simple for all $i$.
\begin{dfn}\label{dfn:length}
Let $\catE$ be an exact category.
\begin{enua}
\item
An object $X$ of $\catE$ is \emph{of finite length}
if the lengths of proper admissible subobject series of $X$ have an upper bound.
\item
$\catE$ is said to be \emph{length} if every object in $\catE$ is of finite length.
\item
A length exact category $\catE$ \emph{satisfies the Jordan-H\"{o}lder property}
if, for every $X\in\catE$, all composition series of $X$ are isomorphic to each other.
\end{enua}
\end{dfn}

Note that any object of finite length has a composition series
since proper admissible subobject series of a maximal length 
are composition series (see \cite[Proposition 2.5]{En22}).

\begin{ex}\label{ex:length ex}
Let $\catE$ be an exact category.
\begin{enua}
\item
A \emph{length-like function} is an additive function $\ell \colon \abs{\catE} \to \bbN$
such that $\ell(X)=0$ implies $X\iso 0$.
If $\catE$ has a length-like function,
then $\catE$ is a length exact category (see \cite[Lemma 4.3]{En22}).
\item
Let $\LL$ be a finite dimensional algebra over a field $\bbk$.
Then $\catmod \LL$ is a length abelian category
since the dimension as vector spaces gives rise to 
a length-like function $\dim_{\bbk} \colon \abs{\catmod \LL} \to \bbN$.
An extension-closed subcategory of $\catmod \LL$ is also a length exact category.
\item
A length abelian category satisfies the Jordan-H\"{o}lder property
(see \cite[p.92, Examples 2]{St}).
\end{enua}
\end{ex}

The following facts are basics to study the Grothendieck monoid of a length exact category.
\begin{fct}[{\cite[Proposition 4.8]{En22}}]\label{fct:length atomic}
Let $\catE$ be a length exact category.
Then $\M(\catE)$ is generated by the set $\{[S] \mid S\in \simp \catE\}$.
Moreover,
$\M(\catE)$ is finitely generated if and only if $\simp \catE$ is a finite set.
\end{fct}

\begin{fct}[{\cite[Theorem 4.12]{En22}}]\label{fct:JHP}
The following are equivalent for an exact category $\catE$.
\begin{enua}
\item
$\catE$ satisfies the Jordan-H\"{o}lder property.
\item
$\M(\catE)$ is a free monoid with a basis $\{[S] \mid S\in \simp \catE\}$.
\end{enua}
In particular,
if $\catA$ is a length abelian category,
then $\M(\catA)$ is a free monoid with a basis $\{[S] \mid S\in \simp \catA\}$.
\end{fct}

\begin{ex}\label{ex:M fd alg}
Let $\LL$ be a finite dimensional algebra over a field $\bbk$.
Then $\M(\catmod \LL)\iso \bbN^{\oplus n}$,
where $n$ is the number of isomorphism classes of simple $\LL$-modules.
The number of maximal right ideals of $\LL$ is also $n$.
Thus if $\LL$ is local, we have $\M(\catmod \LL)\iso\bbN$.
\end{ex}

We will now begin to classify Serre subcategories of a length exact category.
For a subcategory $\catX$ of an exact category $\catE$,
the \emph{Serre subcategory generated by $\catX$} is 
the smallest Serre subcategory $\gen{\catX}_{\Serre}$ containing $\catX$.
A Serre subcategory of the form $\gen{X}_{\Serre}$ is said to be \emph{finitely generated},
where $X\in\catE$.
\begin{prp}\label{prp:sim gen Serre}
Let $\catE$ be a length exact category.
Then we have an inclusion-preserving surjection
\[
\gen{-}_{\Serre} \colon \pow(\simp \catE) \to \Serre(\catE),\quad 
\catX \mapsto \gen{\catX}_{\Serre}.
\]
\end{prp}
\begin{proof}
It follows from
Corollary \ref{cor:Serre bij}, Lemma \ref{lem:gen face} and Fact \ref{fct:length atomic}.
\end{proof}

As a corollary,
we obtain a classification of Serre subcategories 
of an exact category satisfying the Jordan-H\"{o}lder property.
\begin{cor}\label{cor:Serre JHP}
Let $\catE$ be an exact category satisfying the Jordan-H\"{o}lder property.
Then we have an inclusion-preserving bijection
\[
\gen{-}_{\Serre} \colon \pow(\simp \catE) \to \Serre(\catE),\quad 
\catX \mapsto \gen{\catX}_{\Serre}.
\]
\end{cor}
\begin{proof}
It follows from
Example \ref{ex:face free monoid}, Fact \ref{fct:JHP} and Proposition \ref{prp:sim gen Serre}.
\end{proof}

We give a nontrivial example of classifying Serre subcategories of a length exact category
which does not satisfy the Jordan-H\"{o}lder property.
We first introduce the Cayley quiver, 
which is a monoid version of the Cayley graph of a group.
\begin{dfn}[{\cite[Definition 7.5]{En22}}]
Let $M$ be a monoid generated by $A \subseteq M$.
Then the \emph{Cayley quiver} of $M$ with respect to $A$ is a quiver defined as follows:
\begin{itemize}
\item
The vertex set is $M$.
\item
For each $a\in A$ and $m\in M$,
we draw a (labeled) arrow $m\xr{a} m+a$.
\end{itemize}
\end{dfn}
For a length exact category $\catE$,
the natural choice of $A$ above is $\{[S] \mid S\in\simp\catE\}$. 

\begin{ex}[{cf. \cite[Section 7.2]{En22}}]\label{ex:explicit fd}
Let $\LL$ be the path algebra of the quiver $1 \leftarrow 2$ over a field $\bbk$.
Then $\catmod \LL$ is a length abelian category whose indecomposable objects
are exactly two simple modules $S_1, S_2$ and one projective injective module $P$.
Thus $\M(\catmod\LL) = \bbN[S_1]\oplus \bbN[S_2] \iso \bbN^{\oplus 2}$ by Fact \ref{fct:JHP}.
We identify $\M(\catmod\LL)$ with $\bbN^{\oplus 2}$ via this isomorphism.
Set $N:=\bbN (m,n) \subseteq \M(\catmod \LL)$ for $(0,0) \ne (m,n)\in \bbN^{\oplus 2}$.
Consider the extension-closed subcategory $\catD_{N}$ of $\catmod\LL$
corresponding to $N$ (see \S \ref{ss:Serre face}).
Then $\catD_{N}$ is a length exact category by Example \ref{ex:length ex}.
The structure of $\M(\catD_{N})$ is determined by Enomoto \cite[Proposition 7.6]{En22} as follows:
\begin{enua}
\item
$\catD_{N}$ has exactly $l+1$ distinct simple objects $A_0, \dots, A_l$,
where $l:=\min\{m,n\}$ and
\[
A_i:=P^{\oplus i} \oplus S_1^{\oplus (m-i)} \oplus S_2^{\oplus (n-i)}.
\]
Thus $\M(\catD_{N})$ is generated by $[A_0],\dots,[A_l]$.
\item
Set $a_i:=[A_i]$ for $0\le i \le l$.
Then the Cayley quiver of $\M(\catD_{N})$ with respect to $\{a_i\mid 0\le i \le l\}$
is determined as follows,
where $\xr{a_{0 \sim k}}$ denotes $k+1$ arrows $a_0, \dots, a_k$ for $0\le k \le l$.
\begin{enumerate}[label=\textbf{(Case \arabic*)},leftmargin=*]
\item The case $m\ne n$:
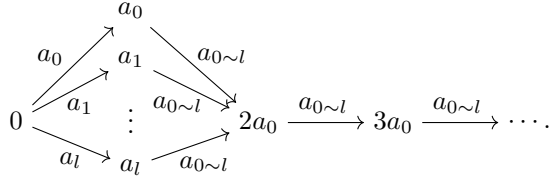
\[
\begin{tikzpicture}[auto]
\node (0) at (0,0) {$0$};
\node[above right =of 0] (a0) {$a_0$};
\node[below ={0.2cm of a0}] (a1) {$a_1$};
\node[below ={0.1cm of a1}] (vdot) {$\vdots$};
\node[below ={0.1cm of vdot}] (al) {$a_l$};
\node[below right =of a0] (2a0) {$2a_0$};
\node[right =of 2a0] (3a0) {$3a_0$};
\node[right =of 3a0] (cdot) {$\cdots$.};

\draw[->] (0) to node[yshift=-3pt] {$a_0$} (a0);
\draw[->] (0) to node[xshift=12pt, yshift=-12pt] {$a_1$} (a1);
\draw[->] (0) to node[xshift=-8pt, yshift=-13pt] {$a_l$} (al);

\draw[->] (a0) to node[xshift=-2pt, yshift=-3pt] {$a_{0 \sim l}$} (2a0);
\draw[->] (a1) to node[xshift=-18pt, yshift=-12pt] {$a_{0 \sim l}$} (2a0);
\draw[->] (al) to node[xshift=18pt, yshift=-13pt] {$a_{0 \sim l}$} (2a0);

\draw[->] (2a0) to node[xshift=-1pt] {$a_{0 \sim l}$} (3a0);
\draw[->] (3a0) to node[xshift=-1pt] {$a_{0 \sim l}$} (cdot);
\end{tikzpicture}
\]
In particular, $\M(\catE)$ is free if and only if either $m=0$ or $n=0$.  
\item The case $m=n$:
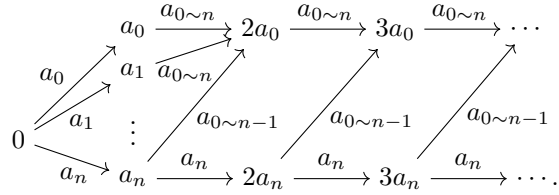
\[
\begin{tikzpicture}[auto]
\node (0) at (0,0) {$0$};

\node[above right =of 0] (a0) {$a_0$};
\node[below ={0.1cm of a0}] (a1) {$a_1$};
\node[below ={0.1cm of a1}] (vdot) {$\vdots$};
\node[below ={0.1cm of vdot}] (al) {$a_n$};

\node[right =of a0] (2a0) {$2a_0$};
\node[right =of 2a0] (3a0) {$3a_0$};
\node[right =of 3a0] (cdot) {$\cdots$};

\node[right =of al] (2al) {$2a_n$};
\node[right =of 2al] (3al) {$3a_n$};
\node[right =of 3al] (cdot2) {$\cdots$.};

\draw[->] (0) to node[yshift=-3pt] {$a_0$} (a0);
\draw[->] (0) to node[xshift=12pt, yshift=-12pt] {$a_1$} (a1);
\draw[->] (0) to node[xshift=-8pt, yshift=-13pt] {$a_n$} (al);

\draw[->] (a0) to node[xshift=-2pt] {$a_{0 \sim n}$} (2a0);
\draw[->] (a1) to node[xshift=11pt, yshift=-15pt] {$a_{0 \sim n}$} (2a0);
\draw[->] (al) to node[xshift=35pt, yshift=-13pt] {$a_{0 \sim n-1}$} (2a0);

\draw[->] (2a0) to node[xshift=-1pt] {$a_{0 \sim n}$} (3a0);
\draw[->] (3a0) to node[xshift=-1pt] {$a_{0 \sim n}$} (cdot);

\draw[->] (al) to node {$a_n$} (2al);
\draw[->] (2al) to node {$a_n$} (3al);
\draw[->] (3al) to node {$a_n$} (cdot2);

\draw[->] (2al) to node[xshift=35pt, yshift=-13pt] {$a_{0 \sim n-1}$} (3a0);
\draw[->] (3al) to node[xshift=35pt, yshift=-13pt] {$a_{0 \sim n-1}$} (cdot);
\end{tikzpicture}
\]
\end{enumerate}
\end{enua}
Now we determine the faces of $\M(\catD_N)$ to classify the Serre subcategories of $\catD_N$:
\begin{description}
\item[(Case 1)]
Any face $F$ of $\M(\catD_N)$ is of the form $\gen{a_i \mid i\in I}_{\face}$ 
for some $I \subseteq \{0,\dots,l\}$ by Lemma \ref{lem:gen face}.
If $I$ is not empty, then $F$ contains $2a_0$.
Thus all $a_i$ belong to $F$ since it is a face, and then $F=\M(\catD_N)$. 
Therefore $\catD_N$ has no nontrivial Serre subcategories.
\item[(Case 2)]
Let $F=\gen{a_i \mid i\in I}_{\face}$ be a face of $\M(\catD_N)$ 
for some $I \subseteq \{0,\dots,n\}$.
If $i\in I$ for $0\le i \le n-1$, then $2a_0 \in F$, and thus $F=\M(\catD_N)$.
Unlike the case $m\ne n$, $\M(\catD_N)$ has a nontrivial face $F=\gen{a_n}_{\face}$.
Hence $\catD_N$ has exactly three Serre subcategories
$0$, $\catD_N$ and $\gen{P^{\oplus n}}_{\Serre}$.
\end{description}
\end{ex}
%
%

\section{The case of exact categories related to a smooth projective curve}\label{s:ex curve}
Hereafter $C$ is a smooth projective curve over an algebraically closed field $\bbk$.
There are three exact categories related to $C$:
\begin{itemize}
\item
The category $\coh(C)$ of coherent sheaves on $C$.
\item
The category $\vect(C)$ of vector bundles over $C$.
\item 
The category $\tor(C)$ of coherent torsion sheaves on $C$.
\end{itemize}
We determine the Grothendieck monoids of them and classify Serre subcategories of them.

For the basics of algebraic geometry, we refer to \cite{Ha,GW}.
We fix notation on schemes.
Let $X$ be a noetherian scheme with structure sheaf $\shO_X$.
A point of $X$ is not necessarily assumed to be closed.
For a point $x\in X$,
we denote by $\mm_x$ the maximal ideal of $\shO_{X,x}$ 
and $\kappa(x):=\shO_{X,x}/\mm_x$ the residue field of $x$.
We denote by $\coh X$ the category of coherent sheaves on $X$.
Let $\shF$ and $\shG$ be coherent sheaves on $X$.
We set $\Hom_{\shO_X}(\shF,\shG):=\Hom_{\coh X}(\shF,\shG)$.
The tensor product of $\shF$ and $\shG$ over $\shO_X$ is denoted by $\shF \otimes_{\shO_X} \shG$.
The sheaf of homomorphisms from $\shF$ to $\shG$ is denoted by $\shHom_{\shO_X}(\shF,\shG)$.
If no ambiguity can arise, we will often omit the subscript $\shO_X$.
The support of $\shF$ is a closed subset of $X$
defined by $\Supp \shF :=\{x\in X \mid \shF_x \ne 0 \}$.
For a noetherian commutative ring $R$,
we identify $\catmod R$ with $\coh(\Spec R)$.
For a morphism $f\colon X \to Y$ of noetherian schemes,
we denote by $f_*\shF$ the direct image of $\shF\in \coh X$
and $f^*\shG$ the pull-back of $\shG \in \coh Y$.
We always have a functor $f^*\colon \coh Y \to \coh X$,
while we have a functor $f_*\colon \coh X \to \coh Y$ when $f$ is proper
(e.g., $f$ is a closed immersion).

We will review the categorical properties of $\coh C$ in each of the following subsections.
Those are well-known and easy to prove
but we write down the proofs since we do not find a suitable reference.
The results of this section are new, except for the lemmas.

\subsection{The case of coherent torsion sheaves}\label{ss:M tor}
We first review a categorical characterization of coherent torsion sheaves on a curve.
Let $i:Z \inj X$ be a closed immersion into a noetherian scheme $X$
and let $\shI$ be the quasi-coherent ideal sheaf corresponding to $Z$.
Then the functor $i_*:\coh Z \to \coh X$ is a fully faithful exact functor
whose essential image $\Ima i_*$ is the subcategory consisting of
coherent sheaves $\shF$ such that $\shI\shF=0$ 
(cf. \cite[\href{https://stacks.math.columbia.edu/tag/01QX}{Tag 01QX}]{SP}).
It follows immediately that $\Ima i_*$ is closed under subobjects in $\coh X$.
This means that there is no difference between
subobjects of $\shF \in \coh Z$ and subobjects of $i_*\shF \in \coh X$.
For a closed point $x\in X$,
consider the natural closed immersion $i \colon \Spec\kappa(x) \inj X$.
Then $\shO_x :=i_*\shO_{\Spec\kappa(x)}$ is a simple object of $\coh X$ by the above discussion.

\begin{lem}\label{lem:char simp sh}
The following are equivalent for a coherent sheaf $\shF$ on a noetherian scheme $X$:
\begin{enua}
\item
$\shF$ is a simple object in $\coh X$.
\item
$\shF\iso \shO_x$ for some closed point $x\in X$. 
\end{enua}
\end{lem}
\begin{proof}
We have already proved that (2) implies (1).
Hence we only prove that (1) implies (2).
Suppose that $\shF$ is a simple object in $\coh X$.
Recall that a simple object is nonzero,
so we have $\Supp \shF \ne \emptyset$.
There is a closed point $x$ of $\Supp \shF$ 
because $X$ is noetherian (cf. \cite[Lemma1.25, Exercise 3.13]{GW}).
Let $i \colon \Spec \kappa(x) \inj X$ be the natural closed immersion.
Then $\shF(x):=i^*\shF=\shF_{x}/\shF_{x}\mm_x \in \catmod \kappa(x)$ 
is nonzero by Nakayama's lemma.
Because the unit morphism $\shF \to i_*i^*\shF =i_*\shF(x)$ is surjective and $\shF$ is simple,
we have that $\shF \isoto i_*\shF(x)$.
Then $\shF(x)$ is also a simple object in $\catmod \kappa(x)$.
This means $\shF(x)\iso \kappa(x)$,
and we obtain the desired result.
\end{proof}

\begin{lem}\label{lem:char fl}
The following are equivalent for a coherent sheaf $\shF$ on a noetherian scheme $X$:
\begin{enua}
\item
$\shF$ is of finite length in $\coh X$ (see Definition \ref{dfn:length}).
\item
$\Supp(\shF)$ consists of only finitely many closed points.
\end{enua}
In this case, the following hold:
\begin{enur}
\item
$\shF_x$ is an $\shO_{X,x}$-module of finite length for any $x\in X$.
\item
The natural morphism $\shF \to \bigoplus_{x\in\Supp \shF} {i_x}_*\shF_x$
is an isomorphism,
where $i_x$ is the natural morphism $\Spec \shO_{X,x} \to X$.
\end{enur}
\end{lem}
\begin{proof}
It is clear when $\shF=0$.
We assume that $\shF\ne0$, and hence $\Supp \shF \ne \emptyset$.

$(1)\imply (2)$:
There is a composition series $0 = \shF_0 \subseteq \shF_1 \subseteq \cdots \subseteq \shF_n=\shF$
in $\coh X$ since $\shF$ is of finite length.
Then $\shF_{i}/\shF_{i-1}\iso \shO_{x_i}$ for some closed point $x_i\in X$ 
by Lemma \ref{lem:char simp sh}.
Thus we have $\Supp \shF = \bigcup_{i=1}^{n} \Supp \shO_{x_i}=\{x_i \mid 1\le i \le n\}$.

$(2)\imply (1)$:
We regard $Z:=\Supp \shF$ as a closed subscheme of $X$
which corresponds to the annihilator $\Ann(\shF)$ of $\shF$ (cf. \cite[Subsection 7.17]{GW}).
Note that $\shO_{Z,x}\iso \shO_{X,x}/\Ann_{\shO_{X,x}}(\shF_x)$ as rings.
Then the natural morphism
$\coprod_{x\in \Supp \shF} \Spec\shO_{Z,x} \to Z$
is an isomorphism 
and $\shO_{Z,x}$ is an artinian local ring
by (2) and \cite[Proposition 5.11]{GW}.
Since $\shF_x$ is finitely generated over the artinian ring $\shO_{Z,x}$,
it is of finite length as an $\shO_{Z,x}$-module,
and hence (i) also holds.
Let $j\colon Z \inj X$ and $j_x\colon \Spec\shO_{Z,x} \to Z$
be the natural closed immersions.
Then we have
\[
j^*\shF 
\iso \bigoplus_{x\in\Supp\shF} {j_x}_*\left(j^*\shF \right)_x
\iso \bigoplus_{x\in\Supp\shF} {j_x}_*\left(\shF_x/\shF_x \cdot \Ann_{\shO_{X,x}}(\shF_x) \right)
=\bigoplus_{x\in\Supp\shF} {j_x}_*\shF_x.
\]
Thus we obtain isomorphisms
\begin{align}\label{eq:char fl}
\shF \isoto j_*j^*\shF \iso j_*\left( \bigoplus_{x\in\Supp\shF} {j_x}_*\shF_x \right) 
\iso \bigoplus_{x\in\Supp \shF} {(jj_x)}_*\shF_x.
\end{align}
See \cite[Remark 7.36]{GW} for the first isomorphism.
Since $jj_x$ is a closed immersion and $\shF_x$ is of finite length,
we conclude that $\shF$ is also of finite length in $\coh X$.
Then (ii) holds by \eqref{eq:char fl} and the following commutative diagram:
\[
\begin{tikzcd}
Z \arr{r,"j",hook} & X\\
\Spec\shO_{Z,x} \arr{u,"j_x",hook} \arr{r,hook} & \Spec\shO_{X,x} \arr{u,"i_x"'}.
\end{tikzcd}
\]
\end{proof}

Let us characterize coherent sheaves of finite length on a smooth projective curve $C$.
For any closed point $x\in C$, 
we set $\shO_{nx}:=i_*\left(\shO_{C,x}/\mm_x^n\right)$,
where $i$ is the natural morphism $\Spec \shO_{C,x} \to C$.
\begin{lem}\label{lem:char tor}
The following are equivalent for a coherent sheaf $\shF$ on $C$:
\begin{enumerate}
\item
$\shF$ is of finite length in $\coh C$.
\item
$\Supp(\shF)$ has only finitely many points.
\item
$\shF_{\eta}=0$ holds, where $\eta$ is the generic point of $C$. 
\end{enumerate}
In this case, the following hold:
\begin{enur}
\item
$\shF_x$ is a torsion $\shO_{C,x}$-module for any $x\in C$.
\item
$\shF \iso \bigoplus_{x\in\Supp \shF} \shO_{n_x x}$ for some positive integers $n_x > 0$.
\end{enur}
\end{lem}
\begin{proof}
It is clear when $\shF=0$.
We assume that $\shF\ne0$, and hence $\Supp \shF \ne \emptyset$.
Since $C$ is a $1$-dimensional integral scheme of finite type over $\bbk$,
the following are equivalent for a non-empty closed subset $Z$ of $C$
(cf. \cite[Proposition 5.20]{GW}):
\begin{itemize}
\item 
$\dim Z = 0$.
\item
$Z$ has only finitely many points.
\item
$Z$ consists of finitely many closed points.
\item
$\eta \not\in Z$.
\end{itemize}
The equivalence of (1), (2) and (3) follows from Lemma \ref{lem:char fl} and the above.
For a finitely generated module over the discrete valuation ring $\shO_{C,x}$,
it is of finite length if and only if it is torsion.
Moreover, it is of the form $\shO_{C,x}/\mm_x^{n_x}$ for some integer $n_x\ge 0$.
Thus (i) and (ii) follow from Lemma \ref{lem:char fl}.
\end{proof}
A coherent sheaf $\shF$ on $C$ is said to be \emph{torsion}
if it satisfies the equivalent conditions of Lemma \ref{lem:char tor}.
We denote by $\tor C$ the category of coherent torsion sheaves.
It is immediate that $\tor C$ is a Serre subcategory of the abelian category $\coh C$.
It is also clear that $\tor C$ is a length abelian category.


We will calculate the Grothendieck monoid $\M(\tor C)$ 
and classify Serre subcategories of $\tor C$.
For this, we recall divisors on $C$.
Let $C(\bbk)$ be the set of closed points of $C$.
We denote by $\Div(C)$ the free abelian group generated by $C(\bbk)$.
An element $D=\sum_{i=1}^n m_i x_i$ of $\Div(C)$ is called a \emph{divisor} on $C$.
The integer $\deg D:=\sum_{i=1}^n m_i$ is called the \emph{degree} of $D$.
A divisor $D=\sum_{i=1}^n m_i x_i$ is said to be \emph{effective} if $m_i\ge 0$ for all $i$.
$\Div^+(C)$ denotes the set of effective divisors on $C$, which is a submonoid of $\Div(C)$.
\begin{prp}\label{prp:M tor}
The following hold.
\begin{enua}
\item
$\simp(\tor C)=\{ \shO_x \mid x \in C(\bbk) \}$ holds 
(see \S \ref{ss:length exact} for the notation).
\item
There is a monoid isomorphism
\[
\Div^+(C) \isoto \M(\tor C),\quad
\sum_{i=1}^n m_i x_i \mapsto  \sum_{i=1}^n m_i [\shO_{x_i}].
\]
\end{enua}
\end{prp}
\begin{proof}
It follows from Lemma \ref{lem:char simp sh} and Fact \ref{fct:JHP}.
\end{proof}

\begin{cor}\label{cor:Serre tor}
There is an inclusion-preserving bijection
\[
\pow(C(\bbk)) \isoto \Serre(\tor C),\quad
A \mapsto \gen{\shO_x \mid x\in A}_{\Serre}.
\]
\end{cor}
\begin{proof}
It follows from
Example \ref{ex:length ex} (3), Corollary \ref{cor:Serre JHP} and Proposition \ref{prp:M tor}.
\end{proof}
Note that $\pow(C(\bbk))$ is exactly the set of specialization-closed subsets except $C$ itself
(see Definition \ref{dfn:spcl order}).
\subsection{The case of vector bundles}\label{ss:M vect}
We begin with a review of vector bundles on a noetherian scheme $X$.
A \emph{locally free sheaf} of rank $n$ on $X$
is a coherent sheaf $\shF$ such that $\shF_x \iso \shO_{X,x}^{\oplus n}$ for all $x\in X$
(cf. \cite[Proposition 7.41]{GW}).
We call a locally free sheaf of finite rank on $X$ a \emph{vector bundle} over $X$.
We denote by $\vect X$ the category of vector bundled over $X$.
Then $\vect X$ is an extension-closed subcategory of $\coh X$.
Indeed, for any exact sequence $0\to \shF \to \shG \to \shH \to 0$ in $\coh X$
with $\shF,\shH \in \vect X$ and any $x\in X$,
the exact sequence $0\to \shF_x \to \shG_x \to \shH_x \to 0$ splits 
since $\shH_x$ is a free $\shO_{X,x}$-module.
Thus $\shG_x \iso \shF_x \oplus \shH_x$ is also a free $\shO_{X,x}$-module for any $x\in X$.
This implies $\shG \in \vect X$, and hence $\vect X$ is extension-closed.
Then $\vect X$ is a length exact category
because the ranks of vector bundles give rise to
a length-like function $\rk \colon \abs{\vect X} \to \bbN$.
An admissible subobject in $\vect X$ is called a \emph{subbundle}.

Before studying the Grothendieck monoid $\M(\vect C)$,
we recall the structure of the Grothendieck group $\K_0(\vect C)$.
For this, we will introduce the Picard group of a noetherian scheme $X$.
A \emph{line bundle} $\shL$ is a vector bundle of rank $1$.
It gives rise to an exact equivalence $-\otimes \shL \colon \coh X \simto \coh X$,
which restricts to an exact equivalence $\vect X \simto \vect X$.
It is clear that $\rk(\shU \otimes \shV)=\rk(\shU)\rk(\shV)$
for any vector bundles $\shU$ and $\shV$.
In particular, we have that $\rk(\shL \otimes \shV)=\rk(\shV)$ if $\shL$ is a line bundle.
The set $\Pic X$ of isomorphism classes of line bundles over $X$ becomes a group
whose operation is the tensor product $\otimes$ and unit is $\shO_X$.
The inverse of $\shL$ in $\Pic(X)$ is given by
the dual $\shL^{\vee}:=\shHom_{\shO_X}(\shL,\shO_X)$ of $\shL$.
The group $\Pic X$ is called the \emph{Picard group} of $X$.
We can assign a vector bundle $\shV$ of rank $r\ge1$
with a line bundle $\det \shV := \bigwedge^{r} \shV$, 
which is called the \emph{determinant bundle} of $\shV$.
We define the determinant bundle of the zero sheaf $0$ by $\det(0):=\shO_X$.
It gives rise to an additive function $\det \colon \abs{\vect X} \to \Pic X$.
\begin{fct}[{\cite[Section 2.6]{LeP}}]\label{fct:K0 vect}
The following holds for a smooth projective curve $C$.
\begin{enua}
\item
The inclusion functor $\vect C \inj \coh C$ induces
a group isomorphism $K_0(\vect C) \isoto \K_0(\coh C)$.
\item
There is a group isomorphism
\[
K_0(\vect C) \isoto \Pic(C) \times \bbZ,\quad
[\shV] \mapsto (\det \shV, \rk \shV).
\]
\end{enua}
\end{fct}

We will determine the Grothendieck monoid $\M(\vect C)$ in Proposition \ref{prp:M vect} below.
Let us give a few preliminaries for Proposition \ref{prp:M vect}.
A coherent sheaf $\shF$ on a noetherian scheme $X$ is \emph{globally generated}
if there exists a surjective morphism $\shO_X^{\oplus n} \surj \shF$.
We do not define \emph{very ample} line bundles which appear in the following fact.
See \cite[Section II.5, page 120]{Ha} for the definition.
We only note that any projective variety has a very ample line bundle.
\begin{fct}[Serre {\cite[Theorem 66.2]{FAC}, cf. \cite[Theorem II.5.17]{Ha}}]\label{fct:ample}
Let $X$ be a projective variety over $\bbk$,
and let $\shO(1)$ be a very ample line bundle on $X$.
Then for any coherent sheaf $\shF$ on $X$,
there is an integer $n_0$ such that $\shF\otimes \shO(1)^{\otimes n}$
is globally generated for all $n \ge n_0$.
\end{fct}

\begin{fct}[{Atiyah \cite[Theorem 2]{Ati57}}]\label{fct:Atiyah}
Let $X$ be a smooth projective variety of dimension $d$ over $\bbk$,
and let $\shV$ be a globally generated vector bundle of rank $r$ over $X$.
If $r > d$,
then $\shV$ contains a trivial subbundle of rank $r-d$,
that is, there is an inflation $\shO_X^{\oplus (r-d)} \infl \shV$ in $\vect X$.
\end{fct}

We prepare notations to use the following proof.
Let $\shO(1)$ be a very ample line bundle on a smooth projective curve $C$.
We set $\shO(n):=\shO(1)^{\otimes n}$ when $n\ge 0$
and $\shO(n):=\left({\shO(1)^{\vee}}\right)^{\otimes n}$ when $n<0$.
For a coherent sheaf $\shF$ on $C$, we set $\shF(n):=\shF\otimes \shO(n)$.
Then $\shF(n) \otimes \shO(m)\iso \shF(n+m)$ holds for any integers $n$ and $m$.
\begin{prp}\label{prp:M vect}
The following hold.
\begin{enua}
\item
A vector bundle is simple in $\vect C$ if and only if
it is a line bundle.
\item
$\M(\vect C)$ is a cancellative monoid,
that is, the natural monoid homomorphism $\M(\vect C) \to \K_0(\vect C)$ is injective
(see Definition \ref{dfn:property mon} and Remark \ref{rmk:compare grp mon}).
\item
There is a monoid isomorphism
\[
\M(\vect C) \isoto \left(\Pic C \times \bbN^+\right) \cup\{(\shO_C,0)\}
\subseteq \Pic C \times \bbZ,\quad
[\shV] \mapsto (\det \shV, \rk \shV),
\]
where $\bbN^{+}:=\bbN\setminus\{0\}$ is the semigroup of strictly positive integers.
\end{enua}
\end{prp}
\begin{proof}
(1)
Let $\shV$ be a vector bundle of rank $r$.
Then there is some integer $n$ such that $\shV(n)$ is globally generated
by Fact \ref{fct:ample}.
If $r>1$,
then there is an inflation $\shO_C^{\oplus (r-1)} \infl \shV(n)$ in $\vect(C)$
by Fact \ref{fct:Atiyah}.
Since the functor $-\otimes \shO(-n)\colon \vect C \simto \vect C$ is exact,
we have an inflation $\shO(-n)^{\oplus (r-1)} \infl \shV$.
Thus a simple object in $\vect C$ has to be a line bundle.
Conversely, a line bundle is a simple object in $\vect C$ 
because $\rk \colon \abs{\vect C} \mapsto \bbN$ is a length-like function.

(2)
Define a monoid homomorphism by $\Phi :=(\det,\rk) \colon \M(\vect C) \to \Pic C \times \bbN$.
Consider the following commutative diagram:
\[
\begin{tikzcd}
\M(\vect C) \arr{d,"\Phi"'} \arr{r} & \K_0(\vect C) \arr{d,"\iso"sloped}\\
\Pic C \times \bbN \arr{r,hook} & \Pic C \times \bbZ.
\end{tikzcd}
\]
It is enough to show that $\Phi$ is injective.
Take vector bundles $\shU$ and $\shV$ such that $\Phi(\shU)=\Phi(\shV)$.
That is, they satisfy $\det \shU\iso \det \shV$ and $r:=\rk\shU = \rk\shV$.
It follows from Fact \ref{fct:ample} that
$\shU(n)$ and $\shV(n)$ are globally generated for some same integer $n$.
Then there are conflations
\[
0\to\shO(-n)^{\oplus (r-1)} \to \shU \to \shL \to 0
\quad\text{and}\quad
0\to\shO(-n)^{\oplus (r-1)} \to \shV \to \shM \to 0
\]
in $\vect C$ by Fact \ref{fct:Atiyah}.
Here $\shL$ and $\shM$ are line bundles.
Then we have
\[
\shL=\det \shL 
\iso \det\shU\otimes\det\left( \shO(-n)^{\oplus (r-1)} \right)^{\vee}
\iso \det\shV\otimes\det\left( \shO(-n)^{\oplus (r-1)} \right)^{\vee}
\iso \det \shM=\shM.
\]
Hence we obtain $[\shU]=[\shL]+(r-1)[\shO(-n)]=[\shM]+(r-1)[\shO(-n)]=[\shV]$ in $\M(\vect C)$.
This proves $\Phi$ is injective.

(3)
It follows from $\Ima\Phi=\left(\Pic C \times \bbN^+\right) \cup\{(\shO_C,0)\}$.
\end{proof}

\begin{cor}\label{cor:Serre vect}
The exact category $\vect C$ has no nontrivial Serre subcategories.
\end{cor}
\begin{proof}
It is enough to show that the monoid
$M:=\left(\Pic C \times \bbN^+\right) \cup\{(\shO_C, 0)\}\subseteq \Pic C \times \bbZ$
has no nontrivial faces by Corollary \ref{cor:Serre bij} and Proposition \ref{prp:M vect} (3).
Let $F$ be a nonzero face of $M$.
There is $(\shL, r) \in F$ such that $(\shL, r) \ne (\shO_C,0)$.
Then we have $(\shO_C,1) \in F$ 
since $2(\shL,r)=(\shL^{\otimes 2},2r-1)+(\shO_C,1)$ in $M$ and $F$ is a face.
For any non-zero element $(\shM, s) \in M$,
we obtain $(\shM,s)+(\shM^{\vee},s)=(\shO_C,2s)=2s(\shO_C,1)\in F$,
and thus $(\shM,s)\in F$.
This means $F=M$, and hence $M$ has no nontrivial faces. 
\end{proof}
\subsection{The case of coherent sheaves}\label{ss:M coh}
We finally deal with the case of the category $\coh C$ of coherent sheaves.
We begin with the relationship between $\tor C$, $\vect C$ and $\coh C$.
\begin{lem}\label{lem:tor pair curve}
The following hold.
\begin{enua}
\item
$\Hom_{\shO_C}(\shT,\shV)=0$ holds for all $\shT\in\tor C$ and $\shV\in\vect C$.
\item
For every coherent sheaf $\shF$ on $C$,
there exists an exact sequence
\[
0\to \shF_{\mathrm{tor}} \to \shF \to \shF_{\mathrm{vect}} \to 0
\]
in $\coh C$ such that $\shF_{\mathrm{tor}} \in \tor C$ and $\shF_{\mathrm{vect}} \in \vect C$.
\end{enua}
In particular, $(\tor C, \vect C)$ is a torsion pair in $\coh C$
(see \cite[Section VI.2]{St} for the definition).
\end{lem}
\begin{proof}
(1)
Let $f\colon \shT \to \shV$ be a morphism 
from a coherent torsion sheaf to a vector bundle  in $\coh C$.
Then $\shT_x$ is a torsion module and $\shV_x$ is a free module for any $x\in C$
by Lemma \ref{lem:char tor} and the definition of vector bundles.
Hence $f_x\colon \shT_x \to \shV_x$ is equal to zero for all $x\in C$.
This implies $f=0$.

(2)
Let $\eta$ be the generic point of $C$ and 
$K(C):=\shO_{C,\eta}$ the function field of $C$.
Consider the natural morphism $j\colon \Spec K(C) \to C$.
Define a coherent sheaf $\shF_{\mathrm{tor}}$ by the kernel of the unit morphism
$\shF \to j_*j^*\shF=j_*\shF_{\eta}$.
Note that $j_*\shF_{\eta}$ is a constant sheaf on $C$ with value $\shF_{\eta}$.
Thus we have $\shF_{\mathrm{tor}}(U)=\{s\in\shF(U) \mid s_{\eta}=0\}$
for every open subset $U$ of $C$.
Then it is clear that $\shF_{\mathrm{tor},\eta}=0$,
and thus $\shF_{\mathrm{tor}}$ is a coherent torsion sheaf.
Set $\shF_{\mathrm{vect}}:=\shF/\shF_{\mathrm{tor}} \in \coh C$.
Then $\shF_{\mathrm{vect}}$ is a subsheaf of the constant sheaf $j_*\shF_{\eta}$.
Hence $\left(\shF_{\mathrm{vect}}\right)_x$ is an $\shO_{C,x}$-submodule of $\shF_{\eta}$ 
for every point $x\in C$.
This implies $\left(\shF_{\mathrm{vect}}\right)_x$ is a torsion-free $\shO_{C,x}$-module,
and thus it is a free $\shO_{C,x}$-module since $\shO_{C,x}$ is a discrete valuation ring.
For this reason, $\shF_{\mathrm{vect}}$ is a vector bundle.
\end{proof}

We will determine the structure of the Grothendieck monoid $\M(\coh C)$ 
in Proposition \ref{prp:M coh} below.
For this, we recall the relation between divisors and line bundles.
We can attach to a divisor $D$ a line bundle $\shO_C(D)$.
It gives rise to a group homomorphism
\[
\Div C \to \Pic C,\quad
D \mapsto \shO_C(D).
\]
For any effective divisor $D = \sum_{i=1}^n n_i x_i$ on $C$,
we set $\shO_D :=\bigoplus_{i=1}^n \shO_{n_i x_i}$
(see the sentence before Lemma \ref{lem:char tor} for the definition of $\shO_{n x}$).
Then there is the following exact sequence in $\coh C$:
\begin{equation}\label{eq:div ses}
0\to \shO_C(-D) \to \shO_C \to \shO_D \to 0.
\end{equation}

Note that the abelian category $\coh C$ is not length 
since there is an infinite subobject series of $\shO_C$:
\[
\cdots \subsetneq \shO_C(-3 x) \subsetneq \shO_C(-2 x) \subsetneq \shO_C(-x)  \subsetneq \shO_C,
\]
where $x$ is a closed point of $C$.
Thus we cannot use the results in \S \ref{ss:length exact}.
\begin{prp}\label{prp:M coh}
The following hold.
\begin{enua}
\item
The inclusion functors $\tor C \inj \coh C$ and $\vect C \inj \coh C$ induce
injective monoid homomorphisms $\M(\tor C) \inj \M(\coh C)$ 
and $\M(\vect C) \inj \M(\coh C)$, respectively.
\item
For any line bundle $\shL$ and any effective divisor $D$,
we have $[\shL]+[\shO_D]=[\shL\otimes \shO_C(D)]$.
\item
$\M(\coh C)$ is the disjoint union of $\M_{\tor C}$ 
and $\M_{\vect C}^+:=\M_{\vect C}\setminus\{0\}$ as a set.
\end{enua}
\end{prp}
\begin{proof}
(1)
The natural monoid homomorphism $\M(\tor C) \to \M(\coh C)$ is injective
by Proposition \ref{prp:inj Serre}.
We prove that the natural monoid homomorphism $\iota \colon \M(\vect C) \to \M(\coh C)$
is injective.
Recall that $\M(\vect C)$ is cancellative
and the natural homomorphism $\K_0(\vect C) \to \K_0(\coh C)$ is an isomorphism
by Proposition \ref{prp:M vect} and Fact \ref{fct:K0 vect}.
It follows that $\iota$ is injective by the following commutative diagram:
\[
\begin{tikzcd}
\M(\vect C) \arr{r,"\iota"} \arr{d,hook} & \M(\coh C) \arr{d}\\
\K_0(\vect C) \arr{r,"\iso"} & \K_0(\coh C).
\end{tikzcd}
\]

(2)
We first note that $\shT \otimes \shL \iso \shT$ for any coherent torsion sheaf $\shT$.
Applying the exact functor $-\otimes(\shL\otimes \shO_C(D)) \colon \coh C \iso \coh C$ 
to the exact sequence \eqref{eq:div ses},
we get an exact sequence
\[
0\to \shL \to \shL\otimes\shO_C(D) \to \shO_D \to 0.
\]
Hence we have the equality $[\shL]+[\shO_D]=[\shL\otimes \shO_C(D)]$.

(3)
For any coherent sheaf $\shF$,
there exists a coherent torsion sheaf $\shT$ and a vector bundle $\shV$
such that $[\shF]=[\shT]+[\shV]$ by Lemma \ref{lem:tor pair curve}.
Then there is an effective divisor $D$ such that $\shT\iso \shO_D$.
We can write $[\shV]=\sum_{i=1}^r [\shL_i]$ for some line bundles $\shL_i$
by Proposition \ref{prp:M vect}.
If $\shV$ is a nonzero vector bundle,
we have
\[
[\shF]=[\shO_D]+\sum_{i=1}^r [\shL_i]
=[\shL_1 \otimes \shO_C(D)] + \sum_{i=2}^r [\shL_i]
=\left[\bigl(\shL_1 \otimes \shO_C(D)\bigr)\oplus \bigoplus_{i=2}^r \shL_i\right] \in \M_{\vect C}.
\]
This proves the desired conclusion.
\end{proof}

As a corollary of Proposition \ref{prp:M coh},
we recover Fact \ref{fct:Gab 1} for smooth projective curves.
See Definition \ref{dfn:spcl order} for the definition of specialization-closed subsets.
We only note here that a specialization-closed subset of $C$
is either a set of closed points or $C$ itself. 
\begin{cor}[{cf. \cite[Proposition VI.2.4]{Gab62}}]
There is an inclusion-preserving bijection between the following sets:
\begin{itemize}
\item
The set of Serre subcategories of $\coh C$.
\item
The set of specialization-closed subsets of $C$.
\end{itemize}
\end{cor}
\begin{proof}
It is enough to classify faces of $\M(\coh C)$ by Corollary \ref{cor:Serre bij}.
Let $F$ be a face of $\M(\coh C)$.
If $[\shV] \in F$ for some nonzero vector bundle $\shV$,
it contains $\M_{\vect C}$ by Corollary \ref{cor:Serre vect}.
Then $F$ must coincide with $\M(\coh C)$ by the exact sequence \eqref{eq:div ses}.
Thus if $F\ne \M(\coh C)$, it is contained in $\M_{\tor C}$.
The faces of $\M(\tor C)$ bijectively correspond to 
the subsets of the set $C(\bbk)$ of closed points 
by Corollary \ref{cor:Serre tor}.
Extending this bijection by assigning $\M(\coh C)$ with $C$,
we obtain the desired bijection.
\end{proof}

Now we compare the Grothendieck monoid $\M(\coh C)$ with the Grothendieck group $\K_0(\coh C)$.
There are unique group homomorphisms $\deg, \rk \colon \K_0(\coh C) \to \bbZ$ 
satisfying the following conditions (see \cite[Section 2.6]{LeP}):
\begin{itemize}
\item 
$\rk(\shF) = \rk(\shF_{\mathrm{vect}})$ for any coherent sheaf $\shF$ on $C$.
\item
$\deg(\shO_C(D))=\deg D$ for any divisor $D$ on $C$.
\item
$\deg(\shO_D)=\deg D$ for any effective divisor $D$ on $C$.
\end{itemize}
The image of the map $(\rk,\deg)\colon \K_0(\coh C) \to \bbZ^{\oplus 2}$ is illustrated as follows:
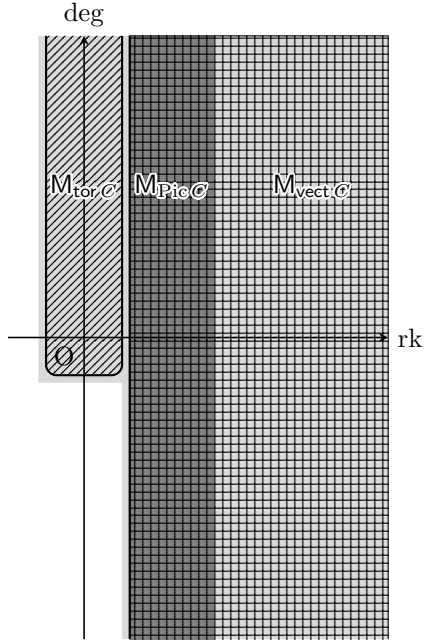
\[
\begin{tikzpicture}[auto]
\fill[gray!30]
(-0.6,-0.6) -- (-0.6,4) -- (4,4) -- (4,-4) -- (0.5,-4) -- (0.5, -0.6) -- cycle;
\fill[gray]
(0.6,-4) -- (0.6,4) -- (1.7,4) -- (1.7,-4);

\draw[thick,black,pattern=north east lines,pattern color=black, rounded corners]
(0.5, 4) -- (0.5, -0.5) -- (-0.5, -0.5) -- (-0.5, 4);
\draw[thick]
(0.6, -4) -- (0.6, 4);
\draw[pattern=grid]
(0.6, -4) -- (0.6, 4) -- (4, 4) -- (4, -4);

\coordinate[label=below left:O] (O) at (0,0); 
\coordinate (XS) at (-1,0); 
\coordinate (XL) at (4,0); 
\coordinate (YS) at (0,-4); 
\coordinate (YL) at (0,4); 
\draw[semithick,->,>=stealth] (XS)--(XL) node[right] {$\rk$}; 
\draw[semithick,->,>=stealth] (YS)--(YL) node[above] {$\deg$}; 

\node (tor) at (0,2) {\contour{white}{$\M_{\tor C}$}};
\node (vect) at (3,2) {\contour{white}{$\M_{\vect C}$}};
\node (pic) at (1.15,2) {\contour{white}{$\M_{\Pic C}$}};
\end{tikzpicture}
\]
Here the gray region corresponds to the Grothendieck monoid $\M(\coh C)$.
Let $\rho \colon \M(\coh C) \to \K_0(\coh C)$ be the natural map.
The map $\rho$ is injective on $\M_{\vect C}$ by Proposition \ref{prp:M vect} and \ref{prp:M coh}.
Whereas, the map $\rho$ loses a lot of information on $\M_{\tor C}$.
Indeed, for two effective divisors $D$ and $E$,
the equality $[\shO_D]=[\shO_E]$ holds in $\K_0(\coh C)$
if and only if $\shO_C( D ) = \shO_C(E)$ in $\Pic C$.

\begin{ex}\label{ex:compare M K_0}
Let $\bbP^1$ be the projective line.
Then $\deg \colon \Pic C \to \bbZ$ is a group isomorphism
(cf. \cite[Example 11.45]{GW}).
In particular,
the map $(\rk,\deg)\colon \K_0(\coh \bbP^1) \to \bbZ^{\oplus 2}$ is a group isomorphism.
For two effective divisors $D$ and $E$ on $\bbP^1$,
the equality $[\shO_D]=[\shO_E]$ holds in $\K_0(\coh C)$
if and only if $\deg D = \deg E$.
Thus the map $\rho$ loses all information except the degrees for torsion sheaves.
In particular, the equality $[\shO_x]=[\shO_y]$ holds in $\K_0(\coh C)$ 
for any closed points $x,y\in\bbP^1(\bbk)$.
Thus the Grothendieck group $\K_0(\coh \bbP^1)$ has no information about closed points of $\bbP^1$.
In contrast,
the Grothendieck monoid $\M(\coh \bbP^1)$ remembers all closed points of $\bbP^1$
because $\M(\coh C) \supseteq \M_{\tor C}= \bigoplus_{x\in \bbP^1(\bbk)} \bbN[\shO_x]$.

This example has another consequence.
Let $\bbk Q$ be the path algebra of Kronecker quiver.
It is well-known that the bounded derived categories $D^b(\coh \bbP^1)$ and $D^b(\catmod \bbk Q)$
are triangulated equivalent.
However, we have a monoid isomorphism $\M(\catmod \bbk Q)\iso \bbN^{\oplus 2}$ 
by Fact \ref{fct:JHP}.
Thus $\M(\coh \bbP^1)$ and $\M(\catmod \bbk Q)$ are not isomorphic as monoids.
This implies the Grothendieck monoids are not derived invariants.
\end{ex}

Finally, we will introduce the notion of the \emph{twisted disjoint union}
to describe the structure of $\M(\coh C)$ in terms of purely monoid-theoretic language.
The rest of this section does not affect the other sections and can be skipped.
We first recall the notion of a monoid action.
Let $M$ be a monoid.
An \emph{$M$-action} on a set $X$ is a monoid homomorphism
$\sigma \colon M\to \End_{\Set}(X):=\Hom_{\Set}(X,X)$.
The pair $X=(X,\sigma)$ is called an \emph{$M$-set}.
Set $\sigma_m :=\sigma(m)$ and $m\cdot x:=\sigma_m(x)$ for all $m\in M$ and $x\in X$.
A map $f\colon X \to Y$ between $M$-sets is \emph{$M$-equivariant}
if $f(m\cdot x)=m\cdot f(x)$ holds for all $m\in M$ and $x\in X$.

Let $X$, $Y$ and $Z$ be $M$-sets.
A map $\alpha \colon X \times Y \to Z$ is an \emph{$M$-bimorphism} if it satisfies
$m\cdot \alpha(x,y)=\alpha(m\cdot x, y)=\alpha(x,m\cdot y)$
for all $m\in M$, $x\in X$ and $y\in Y$.
An \emph{$M$-semigroup} is an $M$-set $S$ with 
an $M$-bimorphism $\alpha \colon S\times S \to S$ satisfying associativity and commutativity.
In other words, it is a (commutative) semigroup $S$ with an $M$-action satisfying
$m\cdot (x+y)=m\cdot x+y =x+ m\cdot y$ for all $m\in M$ and $x,y \in S$.
An \emph{$M$-semigroup homomorphism} is
an $M$-equivariant map $f\colon S \to T$ 
satisfying $f(x+y)=f(x)+f(y)$ for all $x, y\in S$.
We denote by $\SemiGrp_M$ the category of $M$-semigroups and $M$-semigroups homomorphisms.
\begin{ex}\label{ex:homo M-semigrp}
Let $\phi \colon M \to X$ be a monoid homomorphism.
Then $\phi$ defines an action of $M$ on $X$ by $m\cdot x:=\phi(m)+x$ for $m\in M$ and $x\in X$.
We can easily check that $X$ is an $M$-semigroup by this action.
\end{ex}

Let $M$ be a monoid, 
and let $S$ be an $M$-semigroup whose action is given by $\sigma \colon M\to \End_{\Set}(S)$.
The \emph{twisted disjoint union} $M \sqcup_{\sigma} S$ of $M$ and $S$ 
is the set-theoretic disjoint union $M \sqcup S$ with a binary operation given by
\[
x+y:=
\begin{cases}
x+_M y & \text{if both $x\in M$ and $y\in M$,}\\
x+_S y & \text{if both $x\in S$ and $y\in S$,}\\
\sigma_x(y)& \text{if $x\in M$ and $y\in S$,}\\
\sigma_y(x) & \text{if $x\in S$ and $y\in M$,}\\
\end{cases}
\]
where $+_M$ (resp. $+_S$) denotes the binary operation on $M$ (resp. $S$).
We can check easily that $M\sqcup_{\sigma}S$ is a (commutative) monoid.
The natural inclusion $i \colon M\inj M\sqcup_{\sigma}S$ is a monoid homomorphism.
Hence we can think of $M\sqcup_{\sigma} S$ as an $M$-semigroup by Example \ref{ex:homo M-semigrp}.
Then the natural inclusion $j \colon S \inj M\sqcup_{\sigma}S$ is an $M$-semigroup homomorphism.

We describe a universal property of the twisted disjoint union.
We denote by $\Mon_{M/}$ the slice category of $\Mon$ under a monoid $M$.
That is, its objects are monoid homomorphisms $M \to X$,
and morphisms between $\phi\colon M \to X$ and $\psi\colon M \to Y$
are monoid homomorphisms $f\colon X \to Y$ satisfying $f\phi=\psi$.
\begin{prp}\label{prp:UMP tw union}
Let $\phi \colon M \to X$ be a monoid homomorphism.
We regard $X$ as an $M$-semigroup.
Let $S$ be an $M$-semigroup,
and let $i\colon M \to M\sqcup_{\sigma} S$ 
and $j\colon S \to M\sqcup_{\sigma} S$ be the natural inclusions.
Then there is a natural isomorphism
\[
\Hom_{\Mon_{M/}}(M\sqcup_{\sigma} S, X) \isoto \Hom_{\SemiGrp_M}(S,X),\quad
h \mapsto hj.
\]
\end{prp}
\begin{proof}
We omit the proof since it is straightforward.
\end{proof}

Consider the Grothendieck monoid $\M(\coh C)$.
Then $\M(\coh C)$ is an $\M_{\tor C}$-semigroup
by the inclusion homomorphism $\M_{\tor C}\inj \M(\coh C)$.
The subsemigroup $\M_{\vect C}^{+}:=\M_{\vect C}\setminus\{0\}$ is also an $\M_{\tor C}$-semigroup
whose $\M_{\tor C}$-action is given by $\sigma_{[\shO_D]}([\shV]):=[\shO_D]+[\shV]$.
Then the natural inclusion map $\M_{\vect C}^{+} \inj \M(\coh C)$ 
is an $\M_{\tor C}$-semigroup homomorphism.
It induces a monoid homomorphism 
$h \colon \M_{\tor C} \sqcup_{\sigma} \M_{\vect C}^{+} \to \M(\coh C)$
by Proposition \ref{prp:UMP tw union}.
It is clear that $h$ is an isomorphism.
Thus the following statement follows.
\begin{cor}\label{cor:M coh only mon}
There is a monoid isomorphism
\[
\Div^+(C) \sqcup_{\sigma} (\Pic C \times \bbN^+) \isoto \M(\coh C),
\]
where the $\Div^+(C)$-action on $\Pic C \times \bbN^+$ is defined by
$\sigma_D (\shL,r) :=(\shL\otimes \shO_C(D),r)$.
\end{cor}

\section{The spectrum of the Grothendieck monoid}\label{s:spec}
In this section,
we study the monoid spectrum of the Grothendieck monoid $\M(\catE)$ 
and introduce a topology on the set $\Serre(\catE)$ of Serre subcategories.

In \S \ref{ss:spec mon}, we review the spectrum $\MSpec M$ of a monoid $M$ and monoidal spaces,
which are natural analogies of the spectrum of a commutative ring and ringed spaces, respectively.
In \S \ref{ss:spec Gro}, we introduce a topology on $\Serre (\catE)$ for an exact category $\catE$.
We first reveal the relation between the topologies on $\Serre (\catE)$ and $\MSpec \M(\catE)$.
Next, we construct a sheaf $\shM_{\catA}$ of monoids on $\Serre(\catA)$
for an abelian category $\catA$,
which has a property that the stalk $\shM_{\catA,\catS}$ is isomorphic to $\M(\catA/\catS)$
for any point $\catS\in\Serre(\catA)$.
Here $\catA/\catS$ is the abelian quotient category.
We compare $\MSpec \M(\catA)$ with $(\Serre(\catA),\shM_{\catA})$ as monoidal spaces.
In \S \ref{ss:Recon}, we recover the topology of a noetherian scheme $X$ 
from the Grothendieck monoid $\M(\coh X)$.
\subsection{Preliminaries: the spectrum of a commutative monoid}\label{ss:spec mon}
In this subsection,
we review the spectrum of a monoid.
The main reference is \cite{Og}.
We often refer to \cite{GW}, a textbook of scheme theory
since many constructions are analogies of the spectrum of a commutative ring.
Throughout this subsection, $M$ is a monoid.

\begin{dfn}\hfill
\begin{enua}
\item
A subset $I$ of $M$ is called an \emph{ideal}
if for all $x \in I$ and $a \in M$, we have $x+a\in I$.
\item
An ideal $\pp$ of $M$ is said to be \emph{prime}
if it satisfies (i) $\pp \ne M$ and
(ii) $x+y\in \pp$ implies $x\in\pp$ or $y\in \pp$ for all $x,y\in M$.
\item
The \emph{monoid spectrum} of $M$ is the set $\MSpec M$ of prime ideals of $M$.
\end{enua}
\end{dfn}

For a subset $S$ of $M$, the \emph{ideal} $\gen{S}_{\ideal}$ \emph{generated by $S$}
is the smallest ideal of $M$ containing $S$.
We can describe it as follows:
\[
\gen{S}_{\ideal}:=\left\{ x+a \mid x\in S, a\in M  \right\}.
\]

\begin{rmk}\label{rmk:nonempty Spec}\hfill
\begin{enua}
\item
The set $M^+:=M\setminus M^{\times}$ of non-units is the unique maximal ideal of $M$.
It is also a prime ideal of $M$.
\item
The empty set $\emptyset$ is the unique minimal ideal of $M$. 
It is also a prime ideal of $M$.
\item
The monoid spectrum $\MSpec M$ is never empty and 
has the maximum and minimum element with respect to inclusion by (1) and (2).
$\MSpec M$ consists of exactly one point if and only if $M$ is a group.
\end{enua}
\end{rmk}

The relation between faces and prime ideals is the following.
\begin{fct}[{\cite[Section I.1.4]{Og}}]\label{fct:bij face prime}\hfill
\begin{enua}
\item
$\pp^c:=M \setminus \pp$ is a face of $M$
for any prime ideal $\pp$ of $M$.
\item
$F^c:= M \setminus F$ is a prime ideal of $M$
for any face $F$ of $M$.
\item
The assignments given in (1) and (2) give inclusion-reversing bijections between
$\Face(M)$ and $\MSpec M$.
\end{enua}
\end{fct}

We will now endow $\MSpec M$ with the structure of a topological space.
For a subset $S$ of $M$, we set
\[
V(S):=\{\pp \in \MSpec M \mid \pp \supseteq S \}.
\]
Note that $V(S)=V\left( \gen{S}_{\ideal} \right)$ holds.
They satisfy the following equalities(cf. \cite[Lemma 2.1]{GW}):
\begin{itemize}
\item
$V(M)=\emptyset$ and $V(\emptyset)= \MSpec M$.
\item
$\bigcap_{\alpha \in A} V(S_{\alpha}) 
= V\left( \bigcup_{\alpha \in A} S_{\alpha} \right)$
for a family $\{S_{\alpha}\}_{\alpha\in A}$ of subsets of $M$.
\item
$V(I)\cup V(J)=V(I\cap J)$ for ideals $I$, $J$ of $M$.
\end{itemize}
These equalities show that we can define a topology on $\MSpec M$
by taking the subsets of the form $V(S)$ to be the closed subsets.
We call it the \emph{Zariski topology} on $\MSpec M$ .
Note that $\MSpec M$ has a unique closed point $M^+$
and a unique generic point $\emptyset$ by Remark \ref{rmk:nonempty Spec}.
In particular, $\MSpec M$ is an irreducible topological space.

Let
\[
D(f):=\{\pp \in \MSpec M \mid f \not\in \pp \}
\]
for each element $f\in M$.
They are open in $\MSpec M$ since $D(f)=\MSpec M \setminus V(f)$.
They satisfy
\[
D(f) \cap D(g) = D(f+g)
\]
for any $f,g\in M$. 
Open subsets of $\MSpec M$ of this form are called \emph{principal open subsets} of $\MSpec M$.
The set of principal open subsets $D(f)$ forms a basis of the Zariski topology on $\MSpec M$
(cf. \cite[Proposition 2.5]{GW}).

We define a preorder on a topological space.
\begin{dfn}\label{dfn:spcl order}
Let $X$ be a topological space.
\begin{enua}
\item
For two points $x,y \in X$,
we say that $x$ is a \emph{specialization} of $y$ or that $y$ is a \emph{generalization} of $x$
if $x$ belongs to the topological closure $\ol{\{ y\}}$ of $\{y\}$ in $X$. 
Define a preorder $\preceq$ on $X$ by
\[
x\preceq y :\equi \text{$x$ is a specialization of $y$.} 
\]
We call it the \emph{specialization order} on $X$.
When we regard $X$ as a poset by the specialization order,
it is denoted by $X_{\spcl}:=(X, \preceq)$.
\item
A subset $A$ of $X$ is \emph{specialization-closed} (resp. \emph{generalization-closed})
if for any $x\in A$ and every its specialization (resp. generalization) $x'\in X$,
we have that $x'\in A$.
\end{enua}
\end{dfn}
\begin{rmk}\label{rmk:spcl order}
Let $X$ be a topological space.
\begin{enua}
\item
A subset $A$ of $X$ is specialization-closed
if and only if its complement $A^c=X\setminus A$ is generalization-closed.
\item
Any closed subset is specialization-closed.
Also, any open subset is generalization-closed.
\item
Recall that $X$ is called a \emph{$T_0$-space} if for any distinct points,
there exists an open subset containing exactly one of them.
In this case, the specialization order on $X$ is a partial order.
That is, $x\preceq y$ and $y\preceq x$ imply $x=y$ for any $x,y \in X_{\spcl}$.
\end{enua}
\end{rmk}

The specialization order on $\MSpec M$ recovers the inclusion-order on prime ideals.
\begin{prp}\label{prp:spcl MSpec}
The following hold.
\begin{enua}
\item
$\ol{\{\pp\}}=V(\pp)$ for any prime ideal $\pp \subseteq M$.
\item
${\MSpec M}_{\spcl}$ is isomorphic to $\MSpec M$ ordered by reverse inclusion as posets.
\end{enua}
\end{prp}
\begin{proof}
We omit the proof since it is straightforward.
\end{proof}

We give a topological characterization of principal open subsets $D(f)$.
We will use it to classify finitely generated Serre subcategories 
in Proposition \ref{cor:char fg Serre}.
\begin{dfn}
A topological space $X$ is \emph{strongly quasi-compact}
if for every open covering $\{U_i\}_{i\in I}$ of $X$,
there exists $i\in I$ such that $X=U_i$.
\end{dfn}

\begin{lem}\label{lem:sqc MSpec}
Let $M$ be a monoid.
An open subset $U$ of $\MSpec M$ is strongly quasi-compact
if and only if $U=D(f)$ for some $f\in M$.
\end{lem}
\begin{proof}
We first show that $D(f)$ is strongly quasi-compact.
The subset $D(f)$ has the maximum element $\gen{f}_{\face}^c$ with respect to inclusions.
Indeed, for any $\pp \in D(f)$, we have $f\not\in \pp$.
Since $\pp^c$ is a face and $f\in \pp^c$, we obtain $\gen{f}_{\face} \subseteq \pp^c$.
Thus we conclude that $\gen{f}_{\face}^c \supseteq \pp$.
Let $\{U_i\}_{i\in I}$ be an open covering of $D(f)$.
Then there exists $i\in I$ such that $\gen{f}_{\face}^c \in U_i$,
which implies $D(f)=U_i$ because $U_i$ is generalization-closed.
This proves $D(f)$ is strongly quasi-compact.

Conversely,
suppose that $U$ is a strongly quasi-compact open subset of $\MSpec M$.
Since the principal open subsets are a basis of Zariski topology,
the open subset $U$ is covered by them.
Thus $U=D(f)$ for some $f\in M$ because $U$ is strongly quasi-compact.
\end{proof}

Let us recall localization of monoids,
which is a monoid version of localization of commutative rings,
to introduce the structure sheaf on the monoid spectrum $\MSpec M$.
We recommend that the reader skips the remaining part of this subsection
in the first reading.
\begin{dfn}\label{def:mon-loc}
Let $S$ be a subset of $M$.
The \emph{localization of $M$ with respect to $S$} is 
a monoid $M_S$ together with a monoid homomorphism $\rho \colon M\to M_S$,
which is called the \emph{localization homomorphism},
satisfying the following universal property:
\begin{enur}
\item
$\rho(s)$ is invertible for each $s\in S$.
\item
For any monoid homomorphism $\phi \colon M \to X$ 
such that $\phi(s)$ is invertible for each $s\in S$,
there is a unique monoid homomorphism $\ol{\phi} \colon M_S \to X$ satisfying $\phi=\ol{\phi}\rho$.
\end{enur}
\end{dfn}
The localization of $M$ with respect to a subset $S \subseteq M$ always exists.
It is constructed as follows:

Let $\la S \ra_{\bbN}$ be the submonoid of $M$ generated by $S$.
Define an equivalence relation on $M \times \la S \ra_{\bbN}$ by
\[
(x,s) \sim (y,t) :\equi \text{there exist $u\in \la S \ra_{\bbN}$ such that $x+t+u=y+s+u$ in $M$.}
\]
Then the quotient set $M_S:=M\times \la S \ra_{\bbN} /{\sim}$
has a natural monoid structure given by
\[
[x,s]+[y,t]:=[x+y,s+t],
\]
where $[x,s]$ denotes the equivalence class containing $(x,s)\in M\times \la S \ra_{\bbN}$.
We can think of $[x,s]$ as ``$x-s$''.
Then we can check easily that 
the monoid $M_S$ with a monoid homomorphism $\rho \colon M \to M_S$ defined by $\rho(m)=[m,0]$
satisfies the universal property of the localization.


The monoid spectrum $\MSpec M$ is equipped with a natural sheaf of monoids.
\begin{fct}[{\cite[Section II.1.2]{Og}}]
There is a sheaf $\shO_M$ of monoids on $\MSpec M$,
which is called the \emph{structure sheaf},
satisfying the following.
\begin{enua}
\item
For any element $f\in M$,
we have that $\shO_{M}(D(f)) = M_f$.
\item
In particular, we have that $\shO_{M}(\MSpec \M(\catA)) = M$.
\item
For any point $\pp \in \MSpec M$,
the stalk $\shO_{M,\pp}$ of $\shO_{M}$
is isomorphic to the localization $M_{\pp^c}$ of $M$ with respect to $\pp^c:= M\setminus \pp$.
\end{enua}
\end{fct}
A monoidal space is a pair $(X,\shM)$ 
of a topological space $X$ and a sheaf $\shM$ of monoids on $X$.
A morphism $(f,f^{\flat})\colon (X,\shM) \to (Y,\shN)$ of monoidal spaces is
a pair of continuous map $f\colon X\to Y$ and 
a morphism $f^b\colon f^{-1}\shN \to \shM$ of sheaves of monoids
such that the map on the stalks $\shN_{f(x)} \to \shM_{x}$
are local monoid homomorphism for all $x\in X$.
Here a monoid morphism $\phi\colon M\to N$ is \emph{local}
if $\phi^{-1}(N^{\times})= M^{\times}$.
A morphism $(f,f^{\flat})\colon (X,\shM) \to (Y,\shN)$ is an isomorphism if and only if $f$ is a homeomorphism 
and $f^{\flat}$ is an isomorphism of sheaves.
An \emph{affine monoid scheme} is a monoidal space
isomorphic to $(\MSpec M, \shO_M)$ for some monoid $M$.

\begin{rmk}
An affine monoid scheme $(\MSpec M, \shO_{M})$ was first introduced by Kato \cite{Kato}
to study toric singularities.
Deitmar \cite{Deitmar} used it to construct a theory of 
``schemes over the field $\bbF_1$ with one element''.
See \cite{LP11} for more information.
\end{rmk}

The following lemmas will be used in the proof of Proposition \ref{prp:compare monoidal sp}.
\begin{lem}\label{lem:face loc 1}
Let $S$ be a subset of $M$.
Then the natural monoid homomorphism
\[
M_S \to M_{\la S \ra_{\face}},\quad
[x,s] \mapsto [x,s]
\]
is an isomorphism.
\end{lem}
\begin{proof}
Let $\rho\colon M \to M_S$ be the localization homomorphism.
We have that $\rho^{-1}(M_S^{\times}) = \la S \ra_{\face}$ 
by \cite[The text following Proposition 1.4.4]{Og}.
Then the conclusion follows immediately from the universal property.
\end{proof}

\begin{lem}\label{lem:face loc 2}
Let $F$ be a face of $M$.
\begin{enua}
\item
$M/F$ is sharp (see Definition \ref{dfn:property mon} (1)).
\item
The monoid homomorphism
\[
\phi \colon M_F/M_F^{\times} \to M/F,\quad
[x,s] \bmod M_F^{\times} \mapsto x \bmod F
\]
is an isomorphism.
\end{enua}
\end{lem}
\begin{proof}
(1)
Let $x,y\in M$ such that $x+y \equiv 0 \bmod F$.
There are elements $s,t \in F$ such that $x+y+s=t$ in $M$.
Since $F$ is a face, both $x$ and $y$ belong to $F$,
which implies both $x \equiv 0 \bmod F$ and $y \equiv 0 \bmod F$.
Therefore $M/F$ is sharp.

(2)
The quotient homomorphism $M\to M/F$ induces 
a monoid homomorphism $\phi' \colon M_F \to M/F$ by the universal property of $M_F$.
Then $\phi'(M_F^{\times})=0$ since $M/F$ is sharp.
Thus $\phi'$ induces a monoid homomorphism $\phi\colon M_F/M_F^{\times} \to M/F$
by the universal property of $M_F/M_F^{\times}$.
The homomorphism $\phi$ is clearly surjective.
We prove that $\phi$ is injective.
Let $[x,s], [y,t]\in M_F$ such that $x\equiv y \bmod F$.
Then there are $n,n' \in F$ such that $x+n=y+n'$ in $M$.
Hence we have the following equalities in $M_F$:
\[
[x,s]+[s+n,0]=[x+s+n,s]=[x+n,0]=[y+n',0]=[y,t]+[t+n',0].
\]
We conclude that $[x,s]\equiv [y,t] \bmod M_F^{\times}$
because $[s+n,0], [t+n',0] \in M_F^{\times}$.
Therefore $\phi$ is injective.
\end{proof}

\subsection{The spectrum of the Grothendieck monoid}\label{ss:spec Gro}
In this subsection, $\catE$ is an exact category.
We first introduce a topology on the set $\Serre(\catE)$ of Serre subcategories
and study the relationship between the topologies on $\Serre(\catE)$ and $\MSpec \M(\catE)$.
Next, we classify finitely generated Serre subcategories by using this topology.
Finally, we introduce a sheaf $\shM$ of monoids on $\Serre(\catA)$ for an abelian category $\catA$,
which is related to the quotient abelian category $\catA/\catS$,
and compare it with the structure sheaf $\shO_{\M(\catA)}$ of $\MSpec \M(\catA)$.

Let us begin with the bijections which follow from
Corollary \ref{cor:Serre bij} and  Fact \ref{fct:bij face prime}.
\begin{prp}\label{prp:Serre face prime}
There are bijections between the following sets:
\begin{enua}
\item
The set $\Serre (\catE)$ of Serre subcategories of $\catE$.
\item
The set $\Face \M(\catE)$ of faces of $\M(\catE)$.
\item
The set $\MSpec \M(\catE)$ of prime ideals of $\M(\catE)$.
\end{enua}
Moreover,
the bijection between (1) and (2) is inclusion-preserving
while the one between (2) and (3) is inclusion-reversing.
\end{prp}
The bijection between (1) and (3) induces a topology on $\Serre (\catE)$ from $\MSpec \M(\catE)$.
In the following, we describe this topology explicitly.
For a subcategory $\catX$ of $\catE$,
we set
\[
V(\catX):=\{\catS \in \Serre(\catE) \mid \catS \cap \catX = \emptyset \}.
\]
We can easily check that the following equalities hold:
\begin{itemize}
\item 
$V(\catE)=\emptyset$ and $V(\emptyset)=\Serre(\catE)$.
\item
$\bigcap_{\alpha \in A}V(\catX_{\alpha})=V(\bigcup_{\alpha \in A} \catX_{\alpha})$
for a family $\{\catX_{\alpha}\}_{\alpha\in A}$ of subcategories of $\catE$.
\item
$V(\catX) \cup V(\catY)=V(\catX \oplus \catY)$
for subcategories $\catX$, $\catY$ of $\catE$,
where $\catX \oplus \catY :=\{X\oplus Y \mid X\in \catX, Y\in \catY \}$.
\end{itemize}
Thus we can define a topology on $\Serre(\catE)$,
which is called the \emph{Zariski topology},
by taking the subsets of the form $V(\catX)$ to be the closed subsets. 

For an object $X\in \catE$, we put
\[
U_X:=\{\catS \in \Serre (\catE) \mid X \in \catS \}.
\]
We can easily check that $U_X \cap U_Y = U_{X\oplus Y}$ for any $X, Y \in \catE$.

We now compare the Zariski topology on $\MSpec\M(\catE)$ with the one on $\Serre(\catE)$.
\begin{prp}\label{prp:top on Serre}
The following hold.
\begin{enua}
\item
The bijection $\Phi\colon \Serre(\catE) \isoto \MSpec \M(\catE)$ 
in Proposition \ref{prp:Serre face prime} is a homeomorphism.
\item
The set of subsets of the form $U_X$ forms an open basis of $\Serre(\catE)$.
\item
$\Serre(\catE)_{\spcl} \iso (\Serre(\catE),\subseteq)$ as posets 
(see Definition \ref{dfn:spcl order}).
\end{enua}
\end{prp}
\begin{proof}
We first note that $\Phi(\catS)=\M_{\catS}^c:=\M(\catE)\setminus \M_{\catS}$
for any $\catS \in \Serre(\catE)$.

(1)
Let $\catX$ be a subcategory of $\catE$,
and let $\catS$ be a Serre subcategory of $\catE$.
Then $\catS \cap \catX =\emptyset$ if and only if $\M_{\catS} \cap \M_{\catX}=\emptyset$
since $\catS$ is S-closed by Lemma \ref{lem:Serre S-closed}.
It is equivalent to $\Phi(\catS)=\M_{\catS}^c \supseteq \M_{\catX}$.
Thus we obtain $\Phi(V(\catX))=V(\M_{\catX})$,
which implies $\Phi$ is a homeomorphism 
since $\M_{\catX}$ runs through all subsets of $\M(\catE)$ by Proposition \ref{prp:S-closed bij}.

(2)
It is clear since $\Phi(U_X)=D([X])$ for any $X\in \catE$.

(3)
It follows from Proposition \ref{prp:spcl MSpec} and the fact that $\Phi$ is inclusion-reversing.
\end{proof}

\begin{rmk}
The topology on $\Serre(\catE)$ is a natural analogy 
of the topology on the set of thick subcategories of a triangulated category,
which is introduced by Balmer \cite{Balmer} (see also \cite{Matsui-Takahashi, Matsui}).
\end{rmk}

Next, we will classify finitely generated Serre subcategories of $\catE$
by using the Zariski topology on $\Serre (\catE)$.
Recall that a Serre subcategory $\catS$ of $\catE$ is finitely generated 
if $\catS=\gen{X}_{\Serre}$ for some object $X\in \catE$.
We need two lemmas for the open subsets $U_X$ of $\Serre(\catE)$.
\begin{lem}\label{lem:sqc in Serre}
Let $U$ be an open subset of $\Serre(\catE)$.
Then $U$ is strongly quasi-compact
if and only if $U=U_X$ for some $X\in \catE$.
\end{lem}
\begin{proof}
Let $\Phi\colon \Serre(\catE) \isoto \MSpec \M(\catE)$ be the homeomorphism 
in Proposition \ref{prp:top on Serre}.
This lemma immediately follows from Lemma \ref{lem:sqc MSpec} and $\Phi(U_X)=D([X])$.
\end{proof}

\begin{lem}\label{lem:sqc fg Serre}
Let $X$ and $Y$ be objects of $\catE$.
Then $U_X \subseteq U_Y$ if and only if $\gen{X}_{\Serre} \supseteq \gen{Y}_{\Serre}$.
In particular, $U_X = U_Y$ if and only if $\gen{X}_{\Serre} = \gen{Y}_{\Serre}$.
\end{lem}
\begin{proof}
Suppose that $U_X\subseteq U_Y$.
Then $\gen{X}_{\Serre} \in U_X\subseteq U_Y$, which implies $Y\in \gen{X}_{\Serre}$.
Thus we have that $\gen{Y}_{\Serre} \subseteq \gen{X}_{\Serre}$.
Conversely, suppose that $\gen{X}_{\Serre}\supseteq\gen{Y}_{\Serre}$.
Take $\catS \in U_X$.
Then $X\in \catS$, which implies $Y\in \gen{Y}_{\Serre} \subseteq \gen{X}_{\Serre} \subseteq \catS$,
and hence $\catS \in U_Y$.
Thus we have that $U_X \subseteq U_Y$.
\end{proof}

\begin{prp}\label{cor:char fg Serre}
There are bijections between the following sets:
\begin{enua}
\item
The set of finitely generated Serre subcategories of $\catE$.
\item
The set of strongly quasi-compact open subsets of $\Serre (\catE)$.
\item
The set of strongly quasi-compact open subsets of $\MSpec \M(\catE)$.
\end{enua}
The bijection from (1) to (2) is given by $\catX=\gen{X}_{\Serre} \mapsto U_X$.
\end{prp}
\begin{proof}
Let $\Phi\colon \Serre(\catE) \isoto \MSpec \M(\catE)$ be the homeomorphism 
in Proposition \ref{prp:top on Serre}.
It is clear that there is a bijection between (2) and (3) induced by $\Phi$.
Let us construct a bijection between (1) and (2).
For any $X,Y\in\catE$,
$U_X = U_Y$ if and only if $\gen{X}_{\Serre} = \gen{Y}_{\Serre}$ by Lemma \ref{lem:sqc fg Serre}.
Thus the assignment $\catX=\gen{X}_{\Serre} \mapsto U_X$ is well-defined and injective.
On the other hand, it is surjective by Lemma \ref{lem:sqc in Serre}.
Therefore the assignment $\catX=\gen{X}_{\Serre} \mapsto U_X$
gives a bijection from (1) to (2).
\end{proof}

Finally, we construct a sheaf of monoids on $\Serre(\catA)$ for an abelian category $\catA$,
which is related to the quotient abelian category $\catA/\catS$.
There is no application of this sheaf at the moment.
However, it may be interesting from the viewpoints of geometry
over the field $\bbF_1$ with one element and noncommutative algebraic geometry.
Even if the reader skips the rest of this subsection,
there is no problem to read the next subsection.

We begin with a review of the notion of abelian quotient categories.
See \cite[Section 4.3]{Popescu} for details.
For a Serre subcategory $\catS$ of $\catA$,
there are an abelian category $\catA/\catS$ and 
an exact functor $Q\colon \catA \to \catA/\catS$ which satisfy the following universal property:
\begin{itemize}
\item
For any exact functor $F\colon \catA \to \catC$ of abelian categories such that $F(\catS)=0$,
there exists a unique exact functor $\ol{F} \colon \catA/\catS \to \catC$
satisfying $F=\ol{F} Q$.
\end{itemize}
We call $\catA/\catS$ the \emph{abelian quotient category} of $\catA$ with respect to $\catS$,
and $Q\colon \catA \to \catA/\catS$ the \emph{quotient functor}.
The following facts are useful to study the abelian quotient category $\catA/\catS$.
\begin{fct}[{\cite[Lemma 4.3.4, 4.3.7, 4.3.9]{Popescu}}]\label{fct:Serre quot}
Let $\catS$ be a Serre subcategory of an abelian category $\catA$,
and let $Q\colon \catA \to \catA/\catS$ be the quotient functor.
\begin{enua}
\item
$\catS = \{X\in\catA \mid Q(X)=0\}$ holds.
\item
Let $f\colon X\to Y$ be a morphism in $\catA$.
\begin{itemize}
\item
$Q(f)$ is a monomorphism in $\catA/\catS$
if and only if $\Ker(f)\in\catS$.
\item
$Q(f)$ is an epimorphism in $\catA/\catS$
if and only if $\Cok(f)\in\catS$.
\item
$Q(f)$ is an isomorphism in $\catA/\catS$
if and only if $\Ker(f), \Cok(f)\in\catS$.
\end{itemize}
%
\item
Any morphism of $\catA/\catS$ can be written by $Q(s)^{-1}Q(f)Q(t)^{-1}$
for some morphisms $s,t, f$ in $\catA$.
\end{enua}
\end{fct}

Let us construct a sheaf of monoids on $\Serre(\catA)$.
Let $\calB$ be the set of strongly quasi-compact open subsets of $\Serre(\catA)$.
Explicitly, we have that $\calB=\{U_X \mid X\in\catA\}$ by Lemma \ref{lem:sqc in Serre}.
Then $\calB$ is an open basis of $\Serre(\catA)$ by Proposition \ref{prp:top on Serre}. 
Note that 
$U_X \supseteq U_Y$ if and only if $\gen{X}_{\Serre} \subseteq \gen{Y}_{\Serre}$
for any $X,Y\in\catA$ by Lemma \ref{lem:sqc fg Serre}.
In this case,
there is an exact functor $F_{X,Y} \colon \catA/\gen{X}_{\Serre} \to \catA/\gen{Y}_{\Serre}$
induced by the universal property of the abelian quotient category $\catA/\gen{X}_{\Serre}$.
In particular, we obtain a monoid homomorphism 
$r_{X,Y}:=\M(F_{X,Y}) \colon \M(\catA/\gen{X}_{\Serre}) \to \M(\catA/\gen{Y}_{\Serre})$.
Thus the assignment
\[
U_X \mapsto \shM_{\catA}(U_X):=\M(\catA/\gen{X}_{\Serre})
\]
defines a presheaf of monoids on $\calB$.
Define a presheaf $\shM_{\catA}$ on $\Serre(\catA)$ by
\[
V \mapsto \shM_{\catA}(V):=\plim_{U}\shM_{\catA}(U),
\]
where $U$ runs through the set of $U\in\calB$ with $U \subseteq V$.
Then it satisfies the condition of \cite[Proposition 2.20]{GW} 
since $U\in\calB$ is strongly quasi-compact.
Thus $\shM_{\catA}$ is a sheaf on  $\Serre(\catA)$.
We need the following lemma to study this sheaf $\shM_{\catA}$.
\begin{lem}\label{lem:sh on Serre}
Let $\catS$ be a Serre subcategory of an abelian category $\catA$,
and let $X, Y \in \catA$.
\begin{enua}
\item
If $X$ is a subobject of $Y$ in $\catA/\catS$,
then there is $M\in \catS$ such that $X$ remains a subobject of $Y$ in $\catA/\gen{M}_{\Serre}$.
\item
If $Y$ is a quotient of $X$ in $\catA/\catS$,
then there is $M\in \catS$ such that $Y$ remains a quotient of $X$ in $\catA/\gen{M}_{\Serre}$.
\item
If $X\iso Y$ in $\catA/\catS$,
then there is $M\in \catS$ such that $X\iso Y$ in $\catA/\gen{M}_{\Serre}$.
\end{enua}
\end{lem}
\begin{proof}
The proof of (2) is similar to that of (1), and (3) is a consequence of (1) and (2).
Thus we only prove (1).
Any monomorphism $X \inj Y$ in $\catA/\catS$ can be written as $Q(s)^{-1}Q(f)Q(t)^{-1}$
for some morphisms $s,t, f$ in $\catA$ by Fact \ref{fct:Serre quot} (3).
We set
\[
M:=\Ker(s)\oplus \Ker(t) \oplus \Cok(s)\oplus \Cok(t) \oplus \Ker(f).
\]
Then $M\in \catS$ by Fact \ref{fct:Serre quot} (2).
Since $\Ker(s)$, $\Ker(t)$, $\Cok(s)$, $\Cok(t)$ and $\Ker(f)$ belong to $\gen{M}_{\Serre}$,
the morphisms $s$ and $t$ are isomorphisms in $\catA/\gen{M}_{\Serre}$,
and $f$ is a monomorphism in $\catA/\gen{M}_{\Serre}$.
Thus there is a monomorphism $X \inj Y$ in $\catA/\gen{M}_{\Serre}$,
and hence $X$ remains a subobject of $Y$ in $\catA/\gen{M}_{\Serre}$.
\end{proof}

\begin{prp}\label{prp:sh on Serre}
Let $\catA$ be an abelian category
and $\shM_{\catA}$ a sheaf on $\Serre(\catA)$ constructed as above.
\begin{enua}
\item
For any $X\in\catA$,
we have $\shM_{\catA}(U_X) = \M(\catA/\gen{X}_{\Serre})$.
\item
In particular, we have $\shM_{\catA}(\Serre(\catA)) = \M(\catA)$.
\item
For any point $\catS\in\Serre(\catA)$,
the stalk $\shM_{\catA,\catS}$ of $\shM_{\catA}$
is isomorphic to $\M(\catA/\catS)$.
\end{enua}
\end{prp}
\begin{proof}
We only prove (3) because
(1) and (2) are obvious by the definition of $\shM_{\catA}$.
Let $\catS$ be a Serre subcategory of $\catA$.
For any $X\in \catA$ with $\catS \in U_X$,
we have the natural exact functor $\catA/\gen{X}_{\Serre} \to \catA/\catS$.
They induce a monoid homomorphism 
\[
\phi\colon \shM_{\catA,\catS}=\colim_{U_X \ni \catS}\shM_{\catA}(U_X)
=\colim_{U_X \ni \catS}\M(\catA/\gen{X}_{\Serre})
\to \M(\catA/\catS).
\]
It is clear that $\phi$ is surjective.
We now prove $\phi$ is injective.
We first note that the natural map $\M(\catA) \to \shM_{\catA,\catS}$ is surjective
since the natural map $\M(\catA) \to \M(\catA/\gen{M}_{\Serre})$ is surjective.
We denote by $[X]_{\catS}$ the element of $\shM_{\catA,\catS}$ represented by $X\in\catA$.
Suppose that $\phi\left([X]_{\catS} \right)=\phi\left([Y]_{\catS} \right)$
for some $X, Y \in \catA$.
Then $X$ and $Y$ are S-equivalent in $\catA/\catS$ by Fact \ref{fct:= ab}.
Hence there are admissible subobject series 
$0 = X_0 \le X_1 \le \cdots \le X_n=X$ and $0 = Y_0 \le Y_1 \le \cdots \le Y_n=Y$ 
such that $X_i/X_{i-1} \iso Y_{\sg(i)}/Y_{\sg(i)-1}$ in $\catA/\catS$
for some permutation $\sg\in\symg_n$.
Applying Lemma \ref{lem:sh on Serre} to $X_{i-1} \le X_{i}$, $Y_{i-1} \le Y_{i}$ and $X_i/X_{i-1} \iso Y_{\sg(i)}/Y_{\sg(i)-1}$,
and taking their direct sum,
we get $M\in \catS$ such that $X$ and $Y$ remain S-equivalent in $\catA/\gen{M}_{\Serre}$.
Thus $[X]=[Y]$ in $\M(\catA/\gen{M}_{\Serre})$ and $\catS \in U_M$.
This proves $[X]_{\catS}=[Y]_{\catS}$ in $\shM_{\catA,\catS}$.
\end{proof}

We compare the monoidal space $(\Serre(\catA),\shM_{\catA})$ 
with the affine monoid scheme $(\MSpec \M(\catA),\shO_{\M(\catA)})$.
Define a sheaf $\ol{\shO}_{\M(\catA)}$ on $\MSpec \M(\catA)$ by
the sheafification of presheaf
\[
U \mapsto \shO_{\M(\catA)}(U)/\shO_{\M(\catA)}(U)^{\times}.
\]
For any object $X \in \catA$,
we have an isomorphism
\begin{equation}\label{eq:iso 11}
\shO_{\M(\catA)}(D([X]))/\shO_{\M(\catA)}(D([X]))^{\times}
=\M(\catA)_{[X]}/\M(\catA)_{[X]}^{\times}
\isoto \M(\catA)/\gen{[X]}_{\face}
\end{equation}
by Lemma \ref{lem:face loc 1}, \ref{lem:face loc 2}.
In general, the natural monoid homomorphism $\M(\catA)/\M_\catS \to \M(\catA/\catS)$ 
is an isomorphism for any Serre subcategory $\catS$ by \cite[Corollary 4.29]{ES}.
Thus we have an isomorphism
\begin{equation}\label{eq:iso 12}
\M(\catA)/\gen{[X]}_{\face}\isoto \M(\catA/\gen{X}_{\Serre})=\shM_{\catA}(U_X).
\end{equation}
Let $\Phi \colon \Serre(\catE) \isoto \MSpec \M(\catA)$ 
be the homeomorphism in Proposition \ref{prp:top on Serre}.
Combining the isomorphisms \eqref{eq:iso 11} and \eqref{eq:iso 12},
we obtain an isomorphism $\Phi^{-1}\ol{\shO}_{\M(\catA)} \to \shM_\catA$ of sheaves of monoids.
Thus we have the following proposition.
\begin{prp}\label{prp:compare monoidal sp}
The bijection in Proposition \ref{prp:Serre face prime} induces an isomorphism of monoidal spaces
\[
(\Serre(\catA),\shM_{\catA}) \iso (\MSpec \M(\catA),\ol{\shO}_{\M(\catA)}).
\]
\end{prp}
\subsection{Reconstruction of the topology of a noetherian scheme}\label{ss:Recon}
In this subsection, we recover the topology of a noetherian scheme $X$
from the Grothendieck monoid $\M(\coh X)$.
Hereafter $X$ is a noetherian scheme.

We first construct an immersion from $X$ to $\Serre(\coh X)$ as topological spaces.
For any point $x\in X$, define a subcategory of $\coh X$ by
\[
\coh^x X:=\{\shF\in\coh X \mid \shF_x=0 \}.
\]
It is clear that $\coh^x X$ is a Serre subcategory of $\coh X$.
Let $j\colon X \to \Serre(\coh X)$ be a map defined by $j(x):=\coh^x X$.
\begin{lem}\label{lem:emb X}
The map $j\colon X \to \Serre(\coh X)$ is an immersion of topological spaces.
That is, it is a homeomorphism onto a subspace of $\Serre(\coh X)$.
\end{lem}
\begin{proof}
We first prove that $j$ is injective.
Let $x,y \in X$ be distinct points.
Since any scheme is $T_0$-space (cf. \cite[Proposition 3.25]{GW}),
the specialization order on $X$ is a partial order by Remark \ref{rmk:spcl order} (2).
Thus $x\not\in \ol{\{y\}}$ or $y\not\in \ol{\{x\}}$ hold.
We may assume that $x\not\in \ol{\{y\}}$.
Then $\shO_{\ol{\{x\}}} \not\in \coh^x X$ but $\shO_{\ol{\{y\}}} \in \coh^x X$,
where we consider $\ol{\{x\}}$ and $\ol{\{y\}}$ as reduced subschemes of $X$.
Hence $\coh^x X \ne \coh^y X$, which proves $j$ is injective.
For a coherent sheaf $\shF$ on $X$, we have
\[
j^{-1}(U_{\shF})=\{x\in X \mid \shF \in \coh^x X\}
=\{x\in X \mid \shF_x =0\}
=X \setminus \Supp \shF.
\]
Thus $j$ is continuous.
Let $Z$ be a closed subset of $X$.
We consider $Z$ as a reduced subscheme of $X$.
Then it is straightforward that
$\coh^x X \in V(\{\shO_Z\})$ if and only if $x\in \Supp(\shO_Z)=Z$ for any $x\in X$.
Thus we have $j(Z)=j(X)\cap V(\{\shO_Z\})$,
and hence $j(Z)$ is a closed subset of $j(X)$.
Therefore $j$ is a homeomorphism onto the subspace $j(X)$ of $\Serre(\coh X)$.
\end{proof}



Next, we determine the image of the immersion $j\colon X \inj \Serre(\coh X)$.
A Serre subcategory $\catS$ of an abelian category $\catA$ is \emph{meet-irreducible}
if $\catX \cap \catY \subseteq \catS$ implies $\catX \subseteq \catS$ or $\catY \subseteq \catS$
for any $\catX,\catY \in \Serre(\catA)$.
\begin{prp}\label{prp:meet-irred}
For a Serre subcategory $\catS$ of $\coh X$,
it is meet-irreducible if and only if $\catS=\coh^x X$ for some point $x\in X$.
In particular, we have
\[
j(X)=\{\catS \in \Serre(\coh X) \mid \text{$\catS$ is meet-irreducible}\}.
\]
\end{prp}
\begin{proof}
By Gabriel's classification of Serre subcategories (Fact \ref{fct:Gab 1}),
there is a poset isomorphism
\[
\Spcl(X) \isoto \Serre(\coh X),\quad
Z \mapsto \coh_Z X:=\{\shF \in \coh X \mid \Supp\shF \subseteq Z\},
\]
where $\Spcl(X)$ is the set of specialization-closed subsets of $X$ ordered by inclusion.
Then for any $Z\in \Spcl(X)$,
the Serre subcategory $\coh_Z X$ is meet-irreducible if and only if
so is $Z$ in the following sense:
\begin{itemize}
\item A specialization-closed subset $Z$ is \emph{meet-irreducible} if
$A \cap B \subseteq Z$ implies $A \subseteq Z$ or $B\subseteq Z$
for any $A, B \in \Spcl(X)$.
\end{itemize}
We also introduce the dual notion for generalization-closed subsets.
We denote by $\Genl(X)$ the set of generalization-closed subsets of $X$.
\begin{itemize}
\item A generalization-closed subset $U$ is \emph{join-irreducible} if
$A \cup B \supseteq U$ implies $A \supseteq U$ or $B\supseteq U$
for any $A, B \in \Genl(X)$.
\end{itemize}
Then $Z\in \Spcl(X)$ is meet-irreducible if and only $Z^c(:=X\setminus Z) \in \Genl(X)$ is join-irreducible.
Therefore, it is equivalent that determining meet-irreducible Serre subcategories of $\coh X$ 
and determining join-irreducible generalization-closed subsets of $X$.

We now determine join-irreducible generalization-closed subsets of $X$.
For any point $x\in X$,
we denote by $\gen{x}_{\genl}$ the set of generalizations of $x$.
It is clear that
$\gen{x}_{\genl}$ is a join-irreducible generalization-closed subset for any $x\in X$.
We prove that any join-irreducible generalization-closed subset is of the form $\gen{x}_{\genl}$
for some $x\in X$.
Let $A\in \Genl(X)$ be join-irreducible.
For any $x\in A$, there is a minimal element $y\in A$ with respect to the specialization-order
such that $y \preceq x$.
Indeed, 
if $\ol{\{x\}} \cap A$ has exactly one point $x$,
then $x$ itself is a minimal element of $A$.
If $\ol{\{x\}} \cap A$ contains a point $x_1$ such that $x\ne x_1$,
then we have a sequence $x \succeq x_1$ of points of $A$.
Repeating this operation,
we have a sequence $x \succeq x_1 \succeq x_2 \succeq \cdots$ of points of $A$.
This sequence terminates since $X$ is noetherian.
Thus we get a minimal element $y\in A$ such that $y \preceq x$.
Let $I$ be the set of minimal elements of $A$.
Then $A=\bigcup_{a\in I} \gen{a}_{\genl}$ by the discussion above.
Fix $x\in I$.
Since
\[
A=\gen{x}_{\genl} \cup \bigcup_{a\in I, a\ne x} \gen{a}_{\genl}
\]
and $A$ is join-irreducible,
we have that $A=\gen{x}_{\genl}$ or $A=\bigcup_{a\in I, a\ne x} \gen{a}_{\genl}$.
If $A=\bigcup_{a\in I, a\ne x} \gen{a}_{\genl}$,
then $x\in A=\bigcup_{a\in I, a\ne x} \gen{a}_{\genl}$.
Hence there is $a\in I$ such that $a \ne x$ and $x \in \gen{a}_{\genl}$,
which contradicts the minimality of $x$.
Thus we have $A=\gen{x}_{\genl}$.

We have proved that a subset $A$ of $X$ is join-irreducible generalization-closed subsets
if and only if $A=\gen{x}_{\genl}$ for some $x\in X$.
We can easily see that $\shF_x=0$ if and only if $\Supp\shF \subseteq \gen{x}_{\genl}^c$
for any $\shF \in \coh X$ because $\Supp\shF$ is specialization-closed.
Thus $\coh_{\gen{x}_{\genl}^c} X = \coh ^x X$, which proves the proposition.
\end{proof}

Let $X$ be a noetherian scheme.
Define $\Serre(\coh X)_{\irred}$ by the set of meet-irreducible Serre subcategories of $\coh X$.
We consider it as a subspace of $\Serre(\coh X)$.
Then the immersion $j\colon X \isoto \Serre(\coh X)$ induces 
a homeomorphism $X \isoto \Serre(\coh X)_{\irred}$ by Lemma \ref{lem:emb X}.
Thus we can recover the topological space $X$ from the topological space $\Serre(\coh X)$.
In particular, the Grothendieck monoid $\M(\coh X)$ recovers the topology of $X$.
Moreover, $\Serre(\coh X)_{\irred}$ has the following property.
\begin{lem}\label{lem:meet-irred}
Let $X$ and $Y$ be noetherian schemes.
Any homeomorphism $\Serre(\coh X) \isoto \Serre(\coh Y)$ restricts to
a homeomorphism $\Serre(\coh X)_{\irred} \isoto \Serre(\coh Y)_{\irred}$.
\end{lem}
\begin{proof}
Let $\Psi\colon \Serre(\coh X) \isoto \Serre(\coh Y)$ be a homeomorphism.
Then it is clear that 
$\Psi$ is also a poset isomorphism $\Serre(\coh X)_{\spcl} \isoto \Serre(\coh Y)_{\spcl}$.
Thus $\Psi$ is an isomorphism of the poset $\Serre(\coh X)$ and $\Serre(\coh Y)$ 
ordered by inclusion by Proposition \ref{prp:top on Serre} (3).
Therefore $\Psi$ preserves meet-irreducible Serre subcategories,
and hence we have $\Psi(\Serre(\coh X)_{\irred}) = \Serre(\coh Y)_{\irred}$.
\end{proof}

Based on the above considerations, we obtain the following.
\begin{thm}\label{thm:recover}
Consider the following conditions for noetherian schemes $X$ and $Y$.
\begin{enua}
\item
$X \iso Y$ as schemes.
\item
$\M(\coh X) \iso \M(\coh Y)$ as monoids.
\item
$\MSpec \M(\coh X) \iso \MSpec \M(\coh Y)$ as topological spaces.
\item
$\Serre(\coh X) \iso \Serre(\coh Y)$ as topological spaces.
\item
$X \iso Y$ as topological spaces.
\end{enua}
Then ``$(1)\imply (2) \imply (3) \equi (4) \imply (5)$'' hold.
\end{thm}
\begin{proof}
The implications ``$(1)\imply (2) \imply (3)$'' are obvious.
The equivalence ``$(3) \equi (4)$'' follows from Proposition \ref{prp:top on Serre}.
The implication ``$(4)\imply (5)$'' follows from Lemma \ref{lem:emb X} and \ref{lem:meet-irred}.
\end{proof}

We finally comment on \cite{BKS07} and our approach.
\begin{rmk}
Let $X$ be a noetherian scheme.
Buan, Krause and Solberg reconstructed
the topological space $X$ from the poset $\Serre(\coh X)$ of Serre subcategories
in \cite{BKS07}.
We review their approach and compare it with ours.

We first recall the spectrum of a frame.
See \cite{BKS07} and \cite{FandL} for detailed explanations.
A \emph{frame} is a poset $L=(L,\le)$ satisfying the following conditions:
\begin{itemize}
\item 
$L$ is a complete lattice,
that is, any subset $A$ of $L$ admits a supremum $\sup A=\Bor_{a \in A} a$ 
and an infimum $\inf A=\Band_{a \in A} a$.
We denote by $a \lor b:= \sup\{a,b\}$ and $a \land b:= \inf\{a,b\}$.
\item
$L$ satisfies the distributed law:
\[
(\Bor_{a\in A} a)\land b=\Bor_{a\in A} (a\land b)
\]
for any subset $A\subseteq L$ and any element $b\in L$.
\end{itemize}
An element $p$ of a frame $L$ is \emph{meet-irreducible}
if $x \land y \le p$ implies $x\le p$ or $y \le p$ for any $x,y\in L$.
The set $\LSpec L$ of meet-irreducible elements of $L$ 
is called the \emph{lattice spectrum} of $L$.
The set $\LSpec L$ has a topology 
whose closed subsets are of the form
\[
V(a):=\{p\in \LSpec L \mid a \le p\},\quad a\in L.
\]
We can endow the underlying set $\LSpec L$ 
with a new topology by taking subsets of the following form to be the open subsets:
\begin{align}\label{eq:dual top}
Y=\bigcup_{i\in I}Y_i\quad
\text{such that $X \setminus Y_i$ is a quasi-compact open
in $\LSpec L$ for all $i\in I$.}
\end{align}
We denote this new space by $\LSpec^* L$ 
and call this topology the \emph{dual topology} of $\LSpec L$,
which was introduced by Hochster in \cite{Hoch}.

We now come back to the case $L=\Serre(\coh X)$.
Buan, Krause and Solberg proved that
$X$ and $\LSpec^* (\Serre(\coh X))$ are homeomorphic
for any noetherian scheme $X$ in \cite[Section 9]{BKS07}
by using Hochster duality \cite[Proposition 8]{Hoch}.
This gives another proof of ``$(4)\imply (5)$'' in Theorem \ref{thm:recover}.
Indeed, consider the following condition:
\begin{enumerate}
\item[(4.5)] $\Serre(\coh X) \iso \Serre(\coh Y)$ as posets.
\end{enumerate}
Then ``$(4)\imply (4.5)$'' holds by Proposition \ref{prp:top on Serre} (3).
If $\Serre(\coh X) \iso \Serre(\coh Y)$ as posets,
then we have homeomorphisms
\[
X \iso \LSpec^* (\Serre(\coh X)) \iso \LSpec^* (\Serre(\coh Y)) \iso Y.
\]
Thus ``$(4.5)\imply (5)$'' holds.

We now compare our approach with that of \cite{BKS07}.
An advantage of our approach is that we can avoid the dual topology and Hochster duality.
Although $\Serre(\coh X)_{\irred}=\LSpec(\Serre(\coh X))$ as subsets of $\Serre(\coh X)$,
they have different topologies.
$\Serre(\coh X)$ has the correct topology in the sense that
$X$ can be embedded in $\Serre(\coh X)$ as a topological space.

On the other hand,
an advantage of the approach of \cite{BKS07} is that it is more general than ours.
Indeed, we can recover the poset structure of $\Serre(\coh X)$ from 
the topology of $\Serre(\coh X)$ by considering the specialization-order
(Proposition \ref{prp:top on Serre} (3)).
However, the author does not know whether the topology on $\Serre(\coh X)$
which we introduced in this paper can be recovered from the poset structure of $\Serre(\coh X)$.
In particular, we cannot prove ``$(4.5)\imply (5)$'' by our approach.

In summary,
our approach is simpler than \cite{BKS07} and avoids heavy facts such as Hochster duality,
but the approach of \cite{BKS07} is more general than ours.
\end{rmk}



\begin{thebibliography}{NNNN00}
\bibitem[Ati57]{Ati57}
 M.\ F.\ Atiyah,
 \emph{Vector bundles over an elliptic curve},
 Proc.\ London Math.\ Soc.\ \textbf{7} (1957), 414--452. 

\bibitem[Bal05]{Balmer}
 P.\ Balmer,
 \emph{The spectrum of prime ideals in tensor triangulated categories},
 J.\ Reine Angew.\ Math.\ \textbf{588} (2005), 149--168.

\bibitem[BG16]{BG16}
 A.\ Berenstein, J.\ Greenstein,
 \emph{Primitively generated Hall algebras},
 Pacific J.\ Math.\ \textbf{281} (2016), no.\ 2, 287--331.

\bibitem[BO01]{BO01}
 A.\ Bondal, D.\ Orlov,
 \emph{Reconstruction of a variety from the derived category and groups of autoequivalences},
 Comp.\ Math.\ \textbf{125} (2001), 327--344

\bibitem[Bra16]{Br16}
 M.\ Brandenburg,
 \emph{Rosenberg's reconstruction theorem},
 Expo.\ Math.\ \textbf{36} (2018), no.\ 1, 98--117.

\bibitem[Bro97]{Br97}
 G.\ Brookfield,
 \emph{Monoids and Categories of Noetherian Modules},
 Ph.\ D.\ dissertation (1997), 
 University of California, Santa Barbara

\bibitem[Bro98]{Br98}
 G.\ Brookfield,
 \emph{Direct sum cancellation of Noetherian modules},
 J.\ Algebra \textbf{200} (1998), no.\ 1, 207--224.

\bibitem[Bro03]{Br03}
 G.\ Brookfield,
 \emph{The extensional structure of commutative Noetherian rings},
 Comm.\ Algebra \textbf{31} (2003), no.\ 6, 2543--2571.

\bibitem[BKS07]{BKS07}
 A.\ B.\ Buan, H.\ Krause, {\O}.\ Solberg,
 \emph{Support varieties: an ideal approach},
 Homology Homotopy Appl.\ \textbf{9} (2007), no.\ 1, 45--74.

\bibitem[B\"{u}h10]{Bu10}
 T.\ B\"{u}hler,
 \emph{Exact categories},
 Expo.\ Math.\ \textbf{28} (2010), no.\ 1, 1--69.

\bibitem[Dei05]{Deitmar}
 A.\ Deitmar,
 \emph{Schemes over $\bbF_1$},
 Number fields and function fields--two parallel worlds, 87--100,
 Progr.\ Math., \textbf{239}, Birkhäuser Boston, 2005.


\bibitem[Eno22]{En22}
 H.\ Enomoto,
 \emph{The Jordan-H\"{o}lder property and Grothendieck monoids of exact categories},
 Adv.\ Math, \textbf{396} (2022).

\bibitem[ES]{ES}
 H.\ Enomoto, S.\ Saito,
 \emph{Grothendieck monoids of extriangulated categories},
 arXiv:2208.02928.

\bibitem[Gab62]{Gab62}
 P.\ Gabriel,
 \emph{Des cat\'{e}gories ab\'{e}liennes},
 Bull.\ Soc.\ Math.\ France \textbf{90} (1962), 323--448.

\bibitem[GP08]{GP08}
 G.\ Garkusha, M.\ Prest,
 \emph{Reconstructing projective schemes from Serre subcategories},
 J.\ Algebra \textbf{319} (2008), no.\ 3, 1132--1153.

\bibitem[GW20]{GW}
 U.\ Görtz, T.\ Wedhorn,
 \emph{Algebraic geometry I. Schemes--with examples and exercises}, Second edition,
 Springer Studium Mathematik--Master, 2020.

\bibitem[Har77]{Ha}
 R.\ Hartshorne,
 \emph{Algebraic geometry},
 Graduate Texts in Mathematics, \textbf{52},
 Springer-Verlag, 1977.

\bibitem[Her97]{He97}
 I.\ Herzog,
 \emph{The Ziegler spectrum of a locally coherent Grothendieck category},
 Proc.\ London Math.\ Soc.\ \textbf{74} (1997), no.\ 3, 503--558.

\bibitem[Hoc69]{Hoch}
 M.\ Hochster,
 \emph{Prime ideal structure in commutative rings},
 Trans.\ Amer.\ Math.\ Soc.\ \textbf{142} (1969), 43--60.

\bibitem[Kan12]{Ka12}
 R.\ Kanda,
 \emph{Classifying Serre subcategories via atom spectrum},
 Adv.\ Math.\ \textbf{231} (2012), no.\ 3--4, 1572--1588.

\bibitem[Kat94]{Kato}
 K.\ Kato,
 \emph{Toric singularities},
 Amer.\ J.\ Math.\ \textbf{116} (1994), no.\ 5, 1073--1099.

\bibitem[Kin94]{Ki94}
 A.\ D.\ King,
 \emph{Moduli of representations of finite-dimensional algebras},
 Quart.\ J.\ Math.\ Oxford Ser.\ \textbf{45} (1994), no.\ 180, 515--530.

\bibitem[LP11]{LP11}
 J.\ López Peña, O.\ Lorscheid,
 \emph{Mapping $\bbF_1$-land: an overview of geometries over the field with one element},
 Noncommutative geometry, arithmetic, and related topics, 241--265,
 Johns Hopkins Univ. Press, 2011.

\bibitem[LeP97]{LeP}
 J.\ Le Potier,
 \emph{Lectures on vector bundles},
 Cambridge Studies in Advanced Mathematics, \textbf{54},
 Cambridge University Press, 1997.

\bibitem[Mat21]{Matsui}
 H.\ Matsui,
 \emph{Prime thick subcategories and spectra of derived and singularity categories of noetherian schemes},
 Pacific J.\ Math.\ \textbf{313} (2021), no.\ 2, 433--457.

\bibitem[MT20]{Matsui-Takahashi}
 H.\ Matsui, R.\ Takahashi,
 \emph{Construction of spectra of triangulated categories and applications to commutative rings},
 J.\ Math.\ Soc.\ Japan \textbf{72} (2020), no.\ 4, 1283--1307.

\bibitem[Ogu18]{Og}
 A.\ Ogus,
 \emph{Lectures on logarithmic algebraic geometry},
 Cambridge Studies in Advanced Mathematics, \textbf{178},
 Cambridge University Press, 2018.

\bibitem[PP12]{FandL}
 J.\ Picado, A.\ Pultr,
 \emph{Frames and locales. Topology without points},
 Frontiers in Mathematics. Birkhäuser/Springer Basel AG, 2012.

\bibitem[Pop73]{Popescu}
 N.\ Popescu,
 \emph{Abelian categories with applications to rings and modules},
 London Mathematical Society Monographs, No.\ 3. Academic Press, 1973.


\bibitem[Ros98]{Ro98}
 A.\ L.\ Rosenberg,
 \emph{The spectrum of abelian categories and reconstruction of schemes},
 Rings, Hopf algebras, and Brauer groups, 257--274,
 Lecture Notes in Pure and Appl.\ Math.\ \textbf{197}, Dekker, New York, 1998.

\bibitem[Ser55]{FAC}
 J.\ P.\ Serre,
 \emph{Faisceaux algébriques cohérents},
 Ann.\ of Math.\ \textbf{61} (1955), 197--278.

\bibitem[Ses67]{Se67}
 C.\ S.\ Seshadri,
 \emph{Space of unitary vector bundles on a compact Riemann surface},
 Ann.\ of Math.\ \textbf{85} (1967), 303--336.

\bibitem[SP]{SP}
 The {Stacks Project Authors}, 
 \emph{Stacks Project},
 \url{https://stacks.math.columbia.edu},
 2022.

\bibitem[Ste75]{St}
 B.\  Stenstr\"{o}m,
 \emph{Rings of Quotients An Introduction to Methods of Ring Theory},
 Springer-Verlag, 1975.

\bibitem[Zie84]{Zi84}
 M.\ Ziegler,
 \emph{Model theory of modules},
 Ann.\ Pure Appl.\ Logic \textbf{26} (1984), no.\ 2, 149--213.
\end{thebibliography}
\end{document}